\pdfoutput=1
\documentclass[10pt]{amsart}

\usepackage{amssymb, amsthm, amsmath, amsfonts, amsxtra, mathrsfs, mathtools}
\usepackage{tikz-cd}
\usepackage{enumitem}
\setlist[description]{leftmargin=\parindent,labelindent=\parindent}
\usetikzlibrary{matrix,arrows,decorations.pathmorphing}
\usepackage{graphicx}
\usepackage{xcolor}
\usepackage{tikz-cd}
\usepackage{enumitem}
\usepackage{url}
\usepackage{hyperref}
\hypersetup{colorlinks=true,
    linkcolor=blue,
    filecolor=blue,
    citecolor=blue}
\usepackage{float}
\usepackage{colonequals}
\setlist[description]{leftmargin=\parindent,labelindent=\parindent}
\usetikzlibrary{matrix,arrows,decorations.pathmorphing}
\input xy
\xyoption{all}

\usepackage[letterpaper, top=3cm, bottom=3cm,left=2.7cm, right=2.7cm, heightrounded,bindingoffset=0mm, marginparwidth=22mm]{geometry}

\newtheorem{theorem}{Theorem}[section]
\newtheorem{proposition}[theorem]{Proposition}
\newtheorem{lemma}[theorem]{Lemma}

\theoremstyle{definition}

\newtheorem{definition}[theorem]{Definition}

\newtheorem{remark}[theorem]{Remark}

\newtheorem{example}[theorem]{Example}

\def\B{\mathcal{B}}

\def\Mbar{\overline{\mathcal{M}}}
\def\M{\mathcal{M}}
\def\U{\mathcal{U}}

\def\Q{\mathbb{Q}}

\def\ZZ{\mathbb{Z}}

\def\H{\mathcal{H}}

\def\P{\mathbb{P}}
\def\O{\mathcal{O}}
\def\cP{\mathcal{P}}
\def\V{\mathcal{V}}
\def\W{\mathcal{W}}
\def\E{\mathcal{E}}
\def\Zbar{\overline{Z}}

\def\sB{\mathscr{B}}
\def\sU{\mathscr{U}}
\def\sH{\mathscr{H}}
\def\sHbar{\overline{\mathscr{H}}}
\def\sP{\mathscr{P}}

\def\sE{\mathscr{E}}

\newcommand{\pp}{\mathbb{P}}
\newcommand{\qq}{\mathbb{Q}}

\newcommand{\SL}{\mathrm{SL}}
\newcommand{\PGL}{\mathrm{PGL}}
\newcommand{\Mb}{\overline{\mathcal{M}}}
\newcommand{\Sym}{\mathrm{Sym}}
\newcommand{\I}{\mathcal{I}}
\renewcommand{\L}{\mathcal{L}}

\title{Chow Rings of Hurwitz Spaces with Marked Ramification}

\author[E.~Clader]{Emily Clader}
\address{Department of Mathematics, San Francisco State University, United States of America}
\email{eclader@sfsu.edu}

\author[Z.~Hu]{Zhengning Hu}
\address{Department of Mathematics, University of Arizona, United States of America}
\email{zhengninghu@arizona.edu}

\author[H.~Larson]{Hannah Larson}
\address{Department of Mathematics, University of California, Berkeley, United States of America}
\email{hlarson@berkeley.edu}

\author[A.Q.~Li]{Amy Q. Li}
\address{Department of Mathematics, University of Texas at Austin, United States of America}
\email{amyqli@utexas.edu}

\author[R.~Lopez]{Rose Lopez}
\address{Department of Mathematics, University of California, Berkeley, United States of America}
\email{roselopez@berkeley.edu}


\begin{document}

\begin{abstract}
The Hurwitz space $\sHbar_{k,g}$ is a compactification of the space of smooth genus-$g$ curves with a simply-branched degree-$k$ map to $\P^1$.  In this paper, we initiate a study of the Chow rings of these spaces, proving in particular that when $k=3$ (which is the first case in which the Chow ring is not already known), the codimension-2 Chow group is generated by the fundamental classes of codimension-2 boundary strata.  The key tool is to realize the codimension-1 boundary strata of $\sHbar_{3,g}$ as the images of gluing maps whose domains are products of Hurwitz spaces $\sH_{k',g'}(\mu)$ with a single marked fiber of prescribed (not necessarily simple) ramification profile $\mu$, and to prove that the spaces $\sH_{k',g'}(\mu)$ with $k'=2,3$ have trivial Chow ring.
\end{abstract}
\maketitle

\section{Introduction}

The moduli space $\M_g$ of smooth genus-$g$ curves admits a compactification $\Mb_g$ parameterizing stable curves.
Since stable curves are built by gluing together smooth curves at marked points, the boundary of $\Mb_g$ admits a stratification by topological type where each stratum is a finite group quotient of a product of moduli spaces $\M_{g',n'}$. In \cite{CaLa:23}, the authors use this stratification together with a study of the Chow rings of $\M_{g',n'}$ to obtain generators for the Chow rings of $\Mb_g$ for $g \leq 7$. The goal of the present paper is to initiate a similar strategy for Hurwitz spaces and their compactification by moduli of admissible covers.

Let $\sH_{k,g}$ be the moduli space of simply-branched, degree-$k$, genus-$g$ covers of smooth genus-zero curves.
By the Riemann--Hurwitz formula, such covers are branched over $b := 2g - 2 + 2k$ distinct points on the target.
Here, we study the admissible covers compactification $\overline{\sH}_{k,g} \supseteq \sH_{k,g}$, which is the quotient of the moduli space of admissible covers by the $S_b$-action permuting the branch points.
Thus, in our case, the branch divisor defines a morphism $\overline{\sH}_{k,g} \to \Mb_{0,b}/S_b$.
When $k = 2$, this morphism induces an isomorphism on coarse moduli spaces, and so an isomorphism on Chow rings with rational coefficients:
\begin{equation} \label{2} A^*(\overline{\sH}_{2,g}) = A^*(\Mb_{0,b}/S_b) = A^*(\Mb_{0,b})^{S_b}.
\end{equation}

The boundary of $\overline{\sH}_{k,g}$ admits a stratification by the topological type and branching behavior of the source curve.
As a first question towards understanding the intersection theory of $\overline{\sH}_{k,g}$, we may ask: When do the fundamental classes of closures of boundary strata generate the Chow groups $A^i(\overline{\sH}_{k,g})$?  For $k =2$, \eqref{2} implies that the fundamental classes of the closures of boundary strata generate the Chow groups of $\overline{\sH}_{2,g}$ for all codimensions $i$,
because the analogous statement is true on $\Mbar_{0,b}$ by \cite{Kee:92}.
For $i = 1$, the Picard Rank Conjecture predicts that $A^1(\overline{\sH}_{k,g})$ is generated by the fundamental classes of boundary divisors (that is, the closures of codimension-$1$ boundary strata). The conjecture is known to hold for $k \leq 5$ \cite{DePa:15} and $k > g - 1$ \cite{Mul:23}. However, little is known about the higher-codimension Chow groups.
Here, we study the first case of interest, which is codimension $i = 2$ for covers of degree $k = 3$.

\begin{theorem} \label{main}
The Chow group $A^2(\overline{\sH}_{3,g})$ is generated by the fundamental classes of closures of boundary strata of codimension $2$.
\end{theorem}

It was previously shown in \cite{CaLa:22} that $A^*(\sH_{3,g}) = \qq$, so by excision, any codimension-$2$ class on $\overline{\sH}_{3,g}$ is the pushforward of a codimension-$1$ class supported on the boundary.  Thus, the main work in proving Theorem \ref{main} is to show that the codimension-$1$ locally closed boundary strata have trivial Chow groups in codimension $1$. In fact, we will show that each codimension-$1$ boundary stratum of $\overline{\sH}_{3,g}$ has trivial Chow ring. For degree-$3$ covers, each of these codimension-$1$ boundary strata is a finite group quotient of a product of \emph{Hurwitz spaces with marked ramification}, which, for partitions $\mu$ of $k$, are the moduli spaces $\sH_{k',g'}(\mu)$ of degree-$k'$ covers $C \to \pp^1$ that have ramification profile $\mu$ over a marked special fiber and are simply branched elsewhere.  These spaces play a role analogous to $\M_{g',1}$ when studying the codimension-$1$ boundary strata on $\Mb_{g}$ that occur as the images of gluing maps $\M_{g',1} \times \M_{g-g',1} \to \Mb_g$.

In Section \ref{bdivs}, we explicitly describe the codimension-$1$ boundary strata of $\overline{\sH}_{3,g}$: they are all finite group quotients of products of $\sH_{k',g'}(\mu)$ with $k' \leq 3$ and $g' \leq g$. Finite, proper maps induce surjections on Chow groups with rational coefficients, so it suffices to show triviality of the Chow rings of such products.
In general, the Chow ring of a product of spaces need not be generated by the pullbacks of classes from either factor (that is, there is no K\"unneth formula in Chow).  Nevertheless, some particularly nice spaces $X$ satisfy the so-called \emph{Chow--K\"unneth generation Property (CKgP)}, meaning that for any $Y$, the tensor product map
\[A^*(X) \otimes A^*(Y) \to A^*(X \times Y)\]
is surjective. Our main theorem is the following.

\begin{theorem} \label{withramification}
For each partition $\mu \in \{ (3), (2,1), (1,1,1)\}$ and $g\geq0$, we
have $A^*(\sH_{3,g}(\mu))=\Q$ and $\sH_{3,g}(\mu)$ has the Chow--K\"unneth generation Property.
\end{theorem}

In Lemma \ref{h2},
we give a simple argument that $\sH_{2,g}(\mu)$ has trivial Chow ring and the CKgP for $\mu =(2)$ and $g\geq 1$, and for $\mu=(1,1)$ and $g\geq 0$. Combining this with Theorem~\ref{withramification} and the above discussion, we deduce Theorem~\ref{main}.

\subsection{Strategy for Theorem \ref{withramification}}
Our proof of Theorem \ref{withramification}
builds off of ideas used in
\cite{CaLa:22} to compute the Chow ring of $\sH_{3,g}$, which we briefly recall.  For consistency with later notation, we replace $\sH_{3,g}$ and $\sH_{3,g}(\mu)$ with slightly different spaces $\H_{3,g}$ and $\H_{3,g}(\mu)$; explicitly, whereas $\sH_{3,g}$ and $\sH_{3,g}(\mu)$ are quotients of corresponding parameterized Hurwitz spaces by an action of $\PGL_2$, the spaces $\H_{3,g}$ and $\H_{3,g}(\mu)$ are quotients by $\SL_2$.  There are maps
\[\H_{3,g} \rightarrow \sH_{3,g} \;\;\; \text{ and } \;\;\; \H_{3,g}(\mu) \rightarrow \sH_{3,g}(\mu)\]
that are $\mu_2$-gerbes and hence induce isomorphisms on rational Chow rings.  Therefore, for our calculations of rational Chow rings, the distinction between the two spaces is irrelevant; the technical reason why we prefer to work with the latter is explained in Section~\ref{hprime} below.

A geometric point of $\H_{3,g}$ is given by  a degree-$3$ cover $\alpha: C \to \pp^1$.  There is a naturally associated rank-$2$ vector bundle $E_\alpha := (\alpha_*\O_C/\O_{\pp^1})^\vee$ on $\pp^1$, sometimes called the \emph{Tschirnhausen bundle}. The degree of $E_\alpha$ is $g+2$, and there is a natural embedding $C \hookrightarrow \pp E^\vee_\alpha$ via which $C$ can be viewed as the vanishing of a section
\[f \in H^0(\P E_{\alpha}^{\vee}, \gamma^*\det(E_{\alpha}^{\vee}) \otimes \O_{\P E_{\alpha}^{\vee}}(3)),\]
where $\gamma: \P E_{\alpha}^{\vee} \rightarrow \P^1$ denotes the bundle projection.  From this perspective, we have $\alpha = \gamma|_C$, and therefore the data of the degree-$3$ cover $\alpha: C \rightarrow \P^1$ can be recovered from the data of the bundle $E_\alpha$ and the section $f$.  The space of such sections forms a vector bundle $\U$ over an open substack $\B$ of the moduli space $\B_{2,g+2}$ of rank-2, degree-$(g+2)$ vector bundles on $\P^1$-bundles that arise as the projectivization of a rank-2 vector bundle with trivial determinant. Moreover, $\H_{3,g}$ is the open substack of $\U$ on which the vanishing of $f$ is a smooth curve and $\alpha$ has only simple ramification.  Summarizing, we have
\begin{equation*}
\begin{tikzcd}
\H_{3,g} \arrow{r}[swap]{\text{open}} & \U \arrow{d}{\substack{\text{vector}\\ \text{bundle}}} & \\
& \B \arrow{r}[swap]{\text{open}} & \B_{2,g+2},
\end{tikzcd}
\end{equation*}
from which it follows that the Chow ring of $\H_{3,g}$ is generated by pullbacks of classes in the Chow ring of $\B_{2,g+2}$, the latter of which is well-understood by \cite{Lar:23}. The complement of $\H_{3,g} \subseteq \U$ has two components: $\Delta$, consisting of covers with singular source curves, and $T$, consisting of covers with a point of triple ramification (or worse).  By excision, computing $A^*(\H_{3,g})$ is then reduced to computing the image of the pushforward map $A^{*-1}(\Delta \cup T) \to A^*(\U) \cong A^*(\B)$, which was carried out explicitly in \cite{CaLa:22}.

The starting insight of this paper is that we can employ a similar strategy for Hurwitz spaces with marked ramification $\H_{3,g}(\mu)$ by realizing each $\H_{3,g}(\mu)$ as an open substack of a vector bundle $V_\mu$ over an appropriate stack $X_\mu$. In each case, the $X_\mu$ we construct admits a map to $\B$, which we use to obtain a presentation of $A^*(X_\mu)$.  Roughly speaking, $X_\mu$ helps us track the locations of the marked points and $V_\mu$ is the analogue of $\U$ above, parameterizing equations on $\P E^{\vee}$ that cut out a curve for which the map to $\P^1$ has the prescribed ramification at the markings.  These spaces sit in a diagram as below, where the top horizontal map is only defined when $\mu \neq (3)$.
\begin{equation} \label{basic}
\begin{tikzcd}
\H_{3,g} \arrow{d}[swap]{\text{open}} & & \arrow[dashed,swap]{ll}{\text{if $\mu \neq (3)$}} \H_{3,g}(\mu) \arrow{d}{\text{open}} \\
\U \arrow{d}[swap]{\text{vector bundle}} & \U \times_{\B} X_\mu \arrow{d} \arrow{l} & V_\mu  \arrow{l}[swap]{\text{subbundle}} \arrow{dl}{\text{vector bundle}} \\
\B & \arrow{l} X_\mu
\end{tikzcd}
\end{equation}

Since $A^*(\H_{3,g}(\mu))$ is a quotient of $A^*(V_\mu) \cong A^*(X_\mu)$, we thus obtain generators for the Chow ring.  To find the relations, we must study the complement of $\H_{3,g}(\mu)\subseteq V_{\mu}$.  This complement often has several components, which we study using the tools of relative principal parts bundles.  In the cases $\mu = (1,1,1)$ and $\mu = (2,1)$, there is a natural map $\H_{3,g}(\mu) \to \H_{3,g}$ (in fact the top rectangle of \eqref{basic} is a fiber square). It follows that all classes in the kernel of $A^*(\B) \to A^*(\H_{3,g})$ lie in the kernel of $A^*(\B) \to A^*(\H_{3,g}(\mu))$, which simplifies some of our calculations.  The case $\mu = (3)$ is more challenging since there is no such map.  In this case, we have marked a triple ramification point, but we still need to excise covers with triple ramification \emph{elsewhere}. Finding the class of this component of the complement of $\H_{3,g}(\mu) \subseteq V_\mu$ involves a delicate excess intersection calculation.

\subsection{Future directions}

Our proof of Theorem \ref{main} takes advantage of the fact that the boundary divisors of $\overline{\mathscr{H}}_{3,g}$ are finite group quotients of products of Hurwitz spaces parametrizing covers having just one fiber with marked ramification.
When there is just one fiber with marked ramification, each of these spaces can be realized as an open substack of a vector bundle over a well-understood stack, where this vector bundle arises as the kernel of a surjective evaluation map (Section \ref{subsec:principalparts}).
If we try to apply similar methods to study the higher-codimension Chow groups, Hurwitz spaces parameterizing covers with more than one marked fiber necessarily appear. We suspect that when the genus is large compared to the number of fibers with marked ramification, then the appropriate analogues of the evaluation maps used in Section \ref{sec:111}, \ref{sec:21} and \ref{sec:3} are still surjective. However, when the genus is small it appears different arguments may be needed. These low-genus cases will show up as factors in the boundary strata even in higher genus, so they would need to be addressed in any attempt at extending Theorem \ref{main} to higher codimension.

When the degree of the cover is greater than $3$, other difficulties arise.
First we point out that, as discussed in Remark~\ref{rmk:twisted}, the Harris-Mumford moduli space of admissible $k$-covers is no longer normal for $k \geq 4$,
so we must instead consider its normalization, the stack of twisted stable maps into the classifying space $B\mathrm{S_k}$.
For covers of degree $4$ and $5$,
there may be hope of applying similar techniques using the Casnati--Ekedahl structure theorems to relate Hurwitz spaces with marked ramification to better-understood moduli stacks.  However, even for these degrees,
the combinatorial possibilities for the boundary strata increase and, as discussed in Remark~\ref{rem:k>3}, the boundary strata are not necessarily finite group quotients of products of Hurwitz spaces with marked ramification. Although they may be realized as fiber products of such, there is no Chow-K\"unneth generation Property for fiber products in general.

\subsection*{Plan of the paper}
The structure of the paper is the following.
In Section ~\ref{background}, we define Hurwitz spaces that parametrize possibly disconnected covers with marked ramification, and we examine their compactifications.
This allows us to characterize the codimension-1 boundary strata of $\sHbar_{3,g}$.
In Section ~\ref{tools}, we review some properties of Chow groups.
We present a specific formulation of excess intersection which we use in Section ~\ref{sec:3} and also recall the Chow-K\"unneth generation Property.
In addition, we discuss principal parts bundles and evaluation maps, which are the tools that allow us to realize these Hurwitz spaces with marked ramification as open substacks of vector bundles.
In Section ~\ref{previous}, we review previous results on the Chow rings of Hurwitz spaces in degrees 2 and 3.
Here, we explain the technicalities of switching from $\sH_{3,g}(\mu)$ to $\H_{3,g}(\mu)$.
Finally, we calculate relations in $A^*(\H_{3,g}(1,1,1))$ (Section ~\ref{sec:111}), $A^*(\H_{3,g}(2,1))$ (Section ~\ref{sec:21}), and $A^*(\H_{3,g}(3))$ (Section ~\ref{sec:3}) to show that they all have trivial Chow ring.

\subsection*{Acknowledgments} The authors would like to thank the organizers of the 2024 Women in Algebraic Geometry collaborative workshop, and the Institute for Advanced Study for graciously hosting this workshop.
We especially thank Rohini Ramadas for valuable discussions during the workshop.
E.C. was supported by NSF CAREER DMS--2137060.
H.L. was supported by a Clay Research Fellowship.
A.L. was supported in part by NSF DMS–2302475.
R.L. was supported in part by a Simons Collaboration Grant.

\section{Background on Hurwitz spaces}\label{background}

In this section, we explain the relevant background on Hurwitz spaces both with and without marked ramification.

\subsection{Hurwitz spaces}

The Hurwitz space $\sH_{k,g}$ parameterizes maps $\alpha: C \rightarrow \P^1$, where $C$ is a smooth curve of genus $g$ and $\alpha$ is a simply-branched degree-$k$ cover.  That is, there is a finite set $\{y_1, \ldots, y_b\} \subseteq \P^1$ of {\it branch points} of $\alpha$ such that the restriction of $\alpha$ to $\alpha^{-1}(\P^1\setminus \{b_1, \ldots, b_k\})$ is a $k$-sheeted cover and exactly two of these sheets come together over each $y_i$.  A quick calculation with the Riemann--Hurwitz formula shows that
\begin{equation}
    \label{eq:b}
    b = 2g -2 + 2k.
\end{equation}
We will wish to ensure that the target $\P^1$ is stable when marked at the branch points, so we require that
\[2g + 2k \geq 5\]
so that $b \geq 3$.  To put the definition of $\sH_{k,g}$ somewhat more precisely in families, we make the following definition.

\begin{definition}
\label{def:openHurwitz}
Fix $g \geq 0$ and $k \geq 2$ with $2g + 2k \geq 5$.  Define $\sH_{k,g}$ to be the algebraic stack with objects over a scheme $S$ given by diagrams
\begin{equation}
\label{eq:Hurwitzdiagram}
\begin{tikzcd}
C \arrow[rr, "\alpha"] \arrow[rd] &   & P \arrow[ld] \\
                        & S, &
\end{tikzcd}
\end{equation}
in which $P$ and $C$ are connected smooth curves over $S$ of genus zero and $g$, respectively, and $\alpha$ is a simply-ramified degree-$k$ cover.  Morphisms from $(\alpha: C \rightarrow P)$ to $(\alpha': C' \rightarrow P')$ are commuting diagrams
\[
\begin{tikzcd}
    C \arrow[r, "\phi"] \arrow[d, "\alpha"] & C'\arrow[d, "\alpha'"] \\ P \arrow[r,"\sigma"] & P'.
\end{tikzcd}
\]
\end{definition}

  In order to compactify $\sH_{k,g}$, we allow the source curve to degenerate to a nodal curve so long as the target curve degenerates correspondingly.  More precisely, the definition is as follows.

\begin{definition}
\label{def:compactHurwitz}
Fix $g \geq 0$ and $k \geq 2$ with $2g + 2k \geq 5$.  Define $\sHbar_{k,g}$ to be the algebraic stack with objects over a scheme $S$ given by diagrams as in \eqref{eq:Hurwitzdiagram}, where $P$ and $C$ are now connected nodal curves over $S$ of genus zero and $g$, respectively, and, denoting by $P_{\text{ns}}$ and $P_{\text{sing}}$ the nonsingular and singular loci of $P$ over $S$ (and similarly for $C$), we require the following:
\begin{itemize}
\item $\alpha^{-1}(P_{\text{ns}}) = C_{\text{ns}}$, and the restriction of $\alpha$ to this set is a simply-ramified degree-$k$ cover;
\item the target curve $P$, when marked at the smooth branch locus of $\alpha$, is stable;
\item $\alpha^{-1}(P_{\text{sing}}) = C_{\text{sing}}$, and the two branches of any node of $C$ map to $P$ with the same ramification, not necessarily simple.
\end{itemize}
Morphisms between objects in $\sHbar_{k,g}$ are as in Definition~\ref{def:openHurwitz}
\end{definition}

\begin{remark}\label{rmk:twisted}
Readers familiar with admissible covers will recognize the conditions in Definition~\ref{def:compactHurwitz}; indeed, the only difference between the moduli space $\sHbar_{k,g}$ and the Harris--Mumford moduli space $\text{Adm}_{k,g}$ of admissible covers as in \cite{HaMu:82} is that the branch divisor is marked in the latter, so $\sHbar_{k,g}$ is the quotient of $\text{Adm}_{k,g}$ by the $S_b$-action permuting the branch points (where $b$ is determined from $k$ and $g$ by \eqref{eq:b}).  It is worth noting that $\text{Adm}_{k,g}$ is not normal for $k \geq 4$, so in some settings, it is preferable to replace it with the moduli space of {\it twisted admissible covers} as in \cite{ACV:03}, which is the normalization of $\text{Adm}_{k,g}$.  Since the main results of this paper are all for $k \leq 3$, in which Harris--Mumford admissible covers and twisted admissible covers agree, we do not discuss this distinction here.
\end{remark}

In addition to $\sH_{k,g}$, in what follows we will also need the larger space
\[\sH'_{k,g} \supseteq \sH_{k,g}\]
parameterizing covers as in Definition \ref{def:openHurwitz} but in which the ramification is not required to be simple.  This space has a natural stratification according to which ramification profiles occur.  Here, for a partition $\mu = (\mu_1, \ldots, \mu_\ell)$ of $k$, we say that $\alpha: C \rightarrow \P^1$ has {\it ramification profile} $\mu$ over $q \in \P^1$ if there exist points $p_1, \ldots, p_\ell \in C$ such that
\[\alpha^{-1}(q) = \sum_{i=1}^\ell \mu_i p_i\]
as divisors on $C$; simple ramification corresponds to the case $\mu = (2,1,\ldots, 1)$.  The complement $\sH'_{k,g} \setminus \sH_{k,g}$ is the union of the divisor $T$ corresponding to curves with ramification profile $(3,1,\ldots, 1)$ or worse, and the divisor $D$ corresponding to curves with ramification profile $(2,2,1, \ldots, 1)$ or worse (the latter of which is empty for $k \leq 3$).

Even when considering the case of simply-ramified covers, in the compactification $\sHbar_{k,g}$, non-simple ramification can occur at nodes.  Thus, to describe the boundary stratification of $\sHbar_{k,g}$, we introduce Hurwitz spaces that allow for non-simple ramification at prescribed points.

\subsection{Hurwitz spaces with marked ramification}

Hurwitz spaces with marked ramification parameterize possibly disconnected covers of $\pp^1$ with specified ramification behavior above a fixed point on $\pp^1$.  For a cover with $r$ components, we specify the degree and genus of each component, as well as the desired ramification behavior. We encode this numerical data with a partition $\underline{k} = (k_1, \ldots, k_r)$ of $k$, a tuple $\underline{g} = (g_1, \ldots, g_r)$ of non-negative integers, and a collection $\underline{\mu} = (\mu_1, \ldots, \mu_r)$ where $\mu_i = (\mu_i^1, \ldots, \mu_i^{\ell_i})$ is a partition of $k_i$.  Additionally, if $k_i = 1$, we require that $g_i = 0$, since if a cover has degree $1$ on some component it must be an isomorphism onto $\P^1$.

\begin{definition}
\label{def:markedHurwitz}
Given $\underline{k}, \underline{g}$ and $\underline{\mu}$ as above, we define $\sH_{\underline{k}, \underline{g}}(\underline{\mu})$ to be the algebraic stack whose objects over a scheme $S$ are given by tuples of diagrams
\begin{equation}
\label{eq:Hurwitzspacediagram}
\begin{tikzcd}
C_i \arrow{dr} \arrow{rr}{\alpha_i} && P \arrow{dl} \\
& S \arrow[bend left=30]{ul}{p_{i}^j} \arrow[bend right = 30, swap]{ur}{q}
\end{tikzcd}
\end{equation}
for $i =1, \ldots, r$, and $j=1,\ldots, \ell_i$ satisfying the following:
\begin{itemize}
    \item (Pointed target) $P \to S$ is a smooth genus-zero curve over $S$ with a section $q: S \to P$.
    \item (Components of the source) $C_i \to S$ is a smooth genus-$g_i$ curve over $S$ for each $i = 1, \ldots, r$, equipped with a finite, flat, degree-$k_i$ map $\alpha_i: C_i \to P$ such that the inner triangle of \eqref{eq:Hurwitzspacediagram} commutes and $\alpha_i$ is simply-branched away from $Q := q(S) \subseteq P$. Moreover, the branch divisors of all $\alpha_i$ are disjoint on the complement of $Q$.
    \item (Ramification at the markings) $p_{i}^j: S \to C_i$ are sections of $C_i \rightarrow S$ for each $i =1, \ldots, r$ and $j=1, \ldots, \ell_i$ such that the outer triangle of \eqref{eq:Hurwitzdiagram} commutes and the following equality of divisors holds in $C_i$:
\[\alpha_i^{-1}(Q) = \sum_{j=1}^{\ell_i} \mu_i^j P_{i}^j,\]
    where $P_{i}^j = p_{i}^j(S)$ is the image of the section.
    \item (Stability) The union of the supports of the branch divisors of the $\alpha_i$ contains at least two distinct points away from $Q$ in each fiber.
    By Riemann--Hurwitz, this is equivalent to the condition
    \begin{equation} \label{num} \sum_{i=1}^r \left(2g_i - 2 + 2k_i - \sum_{j=1}^{\ell_i} (\mu_i^j -1)\right) \geq 2.\end{equation}
\end{itemize}
A morphism in
$\sH_{\underline{k}, \underline{g}}(\underline{\mu})$ from an object $(P, q, \{C_i\}, \{\alpha_i\}, \{p^j_i\})$
to an object $(P', q', \{C'_i\}, \{\alpha'_i\}, \{{p'}^j_i\})$ over $S$ is
given by the data of morphisms $C_i \to C_i'$ and $P \to P'$ such that everything in the diagram below commutes.
\begin{center}
\begin{tikzcd}
C_i \arrow{d}[swap]{\alpha_i} \arrow{rr} & & C_i' \arrow{d}{\alpha_i'} \\
P \arrow{dr} \arrow{rr} & & P' \arrow{dl} \\
& S \arrow[bend left = 20]{ul}{q} \arrow[bend left = 100]{uul}{p_i^j}
\arrow[bend right = 20, swap]{ur}{q'} \arrow[bend right = 100, swap]{uur}{{p'}_i^j}
\end{tikzcd}
\end{center}
The stability condition ensures that the stabilizers are finite.  If $r = 1$ (meaning the cover is connected), then we write $\sH_{k,g}(\mu)$ instead of $\sH_{\underline{k}, \underline{g}}(\underline{\mu})$.
\end{definition}

Several remarks about this definition are in order.

\begin{remark}
\label{rem:disconnected_to_connected}
When studying degree-3 covers, because of the stability condition, only the cases $r = 1$ and $r = 2$ are possible. (Indeed, if $r = 3$, then $\underline{k} = (1,1,1)$ so $\underline{g} = (0, 0, 0)$ and $\underline{\mu} = ((1),(1),(1))$, which violates \eqref{num}.)
Moreover, there are natural equivalences that show each of the $r = 2$ spaces in degree $3$ is equivalent to an $r = 1$ space in degree $2$:
\[\sH_{(2,1),(g,0)}((2),(1)) \cong \sH_{2,g}((2))
\qquad \text{and} \qquad
\sH_{(2,1),(g,0)}((1,1),(1)) \cong \sH_{2,g}((1,1)).
\]
The first equivalence above is defined by sending $(P,q, (C_1, C_2), (\alpha_1, \alpha_2), p^1_1, p^1_2)$ to $(P, q, C_1,\alpha_1, p^1_1)$, with inverse  given by sending $(P, q, C_1, \alpha_1, p^1_1)$ to $(P, q, (C_1, P), (\alpha_1, \mathrm{id}), p^1_1,q)$. One then checks that the composition of these functors is equivalent to the identity, by using the facts that $\alpha_2: C_2 \to P$ is an isomorphism (since it is degree one) and that $q = \alpha_2 \circ p^1_2$. The second equivalence above is proved similarly.  In spite of these equivalences, it is helpful, when using the Hurwitz spaces with marked ramification to describe the boundary divisors of $\sHbar_{3,g}$, that we view elements as disconnected covers, since the preimage of a component of the target curve $P$ in an element of $\sHbar_{3,g}$ can certainly be disconnected.
\end{remark}

\begin{remark}
\label{rem:k>3}
On the other hand, for $k \geq 4$, not all  all Hurwitz spaces $\sH_{\underline{k},\underline{g}}(\underline{\mu})$ are equivalent to products of Hurwitz spaces with $r = 1$.
For example, $\sH_{(2,2),(1,3)}((2),(2))$ is not equivalent to $\sH_{2,1}((2)) \times \sH_{2,3}((2))$. The reason is that the different components of the source are covers of the \emph{same} genus-zero curve in $\sH_{(2,2),(1,3)}((2),(2))$, whereas
$\sH_{2,1}((2)) \times \sH_{2,3}((2))$ parameterizes covers of possibly different genus-zero curves.
\end{remark}

\begin{remark}
In the case where $r=1$ and $\mu = (1, \ldots, 1)$, the condition of Definition~\ref{def:markedHurwitz} says that $\alpha$ is unramified over $q$, so ``no ramification'' is a valid ramification profile.  Still, in this case, there is additional information in $\sH_{k,g}(\mu)$ beyond what is parameterized by $\sH_{k,g}$, given by a labeling of the fiber of $\alpha$ over $q$.
\end{remark}

Similarly to $\sH'_{k,g}$, we will also sometimes require the larger space
\[\sH'_{\underline{k},\underline{g}}(\underline{\mu}) \supseteq \sH_{\underline{k},\underline{g}}(\underline{\mu})\]
parameterizing the same objects as above, but in which $\alpha$ is not required to have simple ramification away from the markings.

\subsection{Boundary strata}\label{bdivs}

The key reason for our interest in the spaces $\sH_{\underline{k},\underline{g}}(\underline{\mu})$ is the role they play in describing the codimension-$1$ boundary strata on $\sHbar_{k,g}$.  For this discussion, we restrict to the case $k=3$ (which is all that is required for the present work), as certain complications arise when $k \geq 4$ due to the issue mentioned in Remark~\ref{rem:k>3}.

There is a ramified cover
\begin{equation}
    \label{eq:maptoMbar}
    \sHbar_{3,g} \rightarrow \Mbar_{0,b}/S_b
\end{equation}
(where, by \eqref{eq:b}, we have $b = 2g+4$) sending a cover $\alpha: C \rightarrow P$ to the target curve marked at the smooth branch points of $\alpha$.  A (locally closed) {\it codimension-$1$ boundary stratum} of $\Mbar_{0,b}/S_b$ refers to a subset of the form
\[D_j \subseteq \Mbar_{0,b}/S_b\]
for some integer $2 \leq j \leq b-2$, where $D_j$ is the locus of genus-zero curves with precisely two components, one with $j$ marked points and the other with $b-j$ marked points. Note that $D_j = D_{b-j}$. The preimage of $D_j$ under the map \eqref{eq:maptoMbar} may have multiple connected components, and a {\it codimension-$1$ boundary stratum} of $\sHbar_{3,g}$ refers to any of these components.

Let us consider a specific example before describing the codimension-$1$ boundary strata of $\sHbar_{3,g}$ in general.

\begin{example}
\label{ex:BD}
Let $g=4$, so that $b=12$.  We calculate the preimage in $\sHbar_{3,4}$ of the codimension-$1$ boundary stratum $D_7$ in $\Mbar_{0,12}/S_{12}$.  A curve $P$ in $D_7$ has two components $P_1$ and $P_2$, where $P_1$ has seven marked points and $P_2$ has five.

An element of the preimage of $P$ is a map $\alpha: C \rightarrow P$.  If $\alpha^{-1}(P_1)$ is connected, then applying the Riemann--Hurwitz formula to the restriction of $\alpha$ to this component (and recalling that all seven of the smooth branch points of $P_1$ are simply-ramified) shows that the genus $g_1$ of the preimage of $P_1$ satisfies
\begin{equation}
    \label{eq:RHexample}
2g_1 - 2 = 3 \cdot (-2) + 7 + \sum_{i=1}^\ell (\mu_i-1),
\end{equation}
where $\mu = (\mu_1, \ldots, \mu_\ell)$ is the ramification profile over the node of $P$.  Simplifying shows that $\sum (\mu_i - 1)$ is odd, and since $\mu$ is a partition of $3$, this is only possible if $\mu = (2,1)$.  From here, solving equation \eqref{eq:RHexample} shows $g_1 = 2$.  A similar calculation shows that, if $\alpha^{-1}(P_2)$ is connected, then the ramification profile over the node must also be $\mu=(2,1)$ and its genus must be $g_2 =1$.

Since $C$ is connected, it is straightforward to check in this case that at least one of $\alpha^{-1}(P_1)$ or $\alpha^{-1}(P_2)$ must be connected; thus, if $\alpha^{-1}(P_i)$ is disconnected, then it consists of a component mapping with degree $1$ and a component mapping with degree $2$.  From here, one sees that there are three possible geometries of the curve $C$, depending on whether  only $\alpha^{-1}(P_1)$, only $\alpha^{-1}(P_2)$, or both are connected (here, the labels on each half-node show the ramification indices):

\vspace{0.25cm}

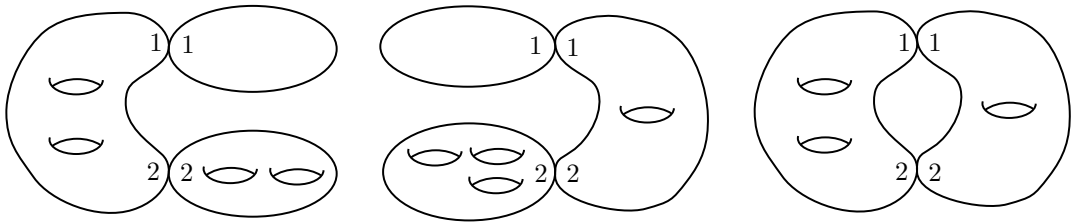
\tikzset{every picture/.style={line width=0.75pt}} 
\begin{center}
\begin{minipage}{0.3\linewidth}
\begin{tikzpicture}[x=0.75pt,y=0.75pt,yscale=-0.7,xscale=0.7]

\draw   (166,88.8) .. controls (180,77.8) and (201,75.8) .. (224,75.8) .. controls (247,75.8) and (265,97.8) .. (249,111.8) .. controls (233,125.8) and (224,123.8) .. (224,140.8) .. controls (224,157.8) and (239,167.8) .. (250,177.8) .. controls (261,187.8) and (254,209.8) .. (229,217.8) .. controls (204,225.8) and (170,211.61) .. (158,194.8) .. controls (146,178) and (138.14,169.07) .. (138,147.8) .. controls (137.86,126.54) and (152,99.8) .. (166,88.8) -- cycle ;
\draw  [draw opacity=0] (207.51,124.31) .. controls (207.51,124.31) and (207.51,124.31) .. (207.51,124.31) .. controls (207.51,124.31) and (207.51,124.31) .. (207.51,124.31) .. controls (207.51,130) and (198.92,134.61) .. (188.31,134.61) .. controls (177.71,134.61) and (169.11,130) .. (169.11,124.31) .. controls (169.11,123.74) and (169.2,123.19) .. (169.36,122.65) -- (188.31,124.31) -- cycle ; \draw   (207.51,124.31) .. controls (207.51,124.31) and (207.51,124.31) .. (207.51,124.31) .. controls (207.51,124.31) and (207.51,124.31) .. (207.51,124.31) .. controls (207.51,130) and (198.92,134.61) .. (188.31,134.61) .. controls (177.71,134.61) and (169.11,130) .. (169.11,124.31) .. controls (169.11,123.74) and (169.2,123.19) .. (169.36,122.65) ;
\draw  [draw opacity=0] (171.06,128.85) .. controls (175.45,125.64) and (181.86,123.61) .. (189,123.61) .. controls (195.47,123.61) and (201.34,125.28) .. (205.65,127.98) -- (189,139.21) -- cycle ; \draw   (171.06,128.85) .. controls (175.45,125.64) and (181.86,123.61) .. (189,123.61) .. controls (195.47,123.61) and (201.34,125.28) .. (205.65,127.98) ;

\draw  [draw opacity=0] (207.51,167.31) .. controls (207.51,167.31) and (207.51,167.31) .. (207.51,167.31) .. controls (207.51,167.31) and (207.51,167.31) .. (207.51,167.31) .. controls (207.51,173) and (198.92,177.61) .. (188.31,177.61) .. controls (177.71,177.61) and (169.11,173) .. (169.11,167.31) .. controls (169.11,166.74) and (169.2,166.19) .. (169.36,165.65) -- (188.31,167.31) -- cycle ; \draw   (207.51,167.31) .. controls (207.51,167.31) and (207.51,167.31) .. (207.51,167.31) .. controls (207.51,167.31) and (207.51,167.31) .. (207.51,167.31) .. controls (207.51,173) and (198.92,177.61) .. (188.31,177.61) .. controls (177.71,177.61) and (169.11,173) .. (169.11,167.31) .. controls (169.11,166.74) and (169.2,166.19) .. (169.36,165.65) ;
\draw  [draw opacity=0] (171.06,171.85) .. controls (175.45,168.64) and (181.86,166.61) .. (189,166.61) .. controls (195.47,166.61) and (201.34,168.28) .. (205.65,170.98) -- (189,182.21) -- cycle ; \draw   (171.06,171.85) .. controls (175.45,168.64) and (181.86,166.61) .. (189,166.61) .. controls (195.47,166.61) and (201.34,168.28) .. (205.65,170.98) ;

\draw   (255,191.8) .. controls (255,174.68) and (282.09,160.8) .. (315.5,160.8) .. controls (348.91,160.8) and (376,174.68) .. (376,191.8) .. controls (376,208.93) and (348.91,222.8) .. (315.5,222.8) .. controls (282.09,222.8) and (255,208.93) .. (255,191.8) -- cycle ;
\draw   (255,101.8) .. controls (255,84.68) and (282.09,70.8) .. (315.5,70.8) .. controls (348.91,70.8) and (376,84.68) .. (376,101.8) .. controls (376,118.93) and (348.91,132.8) .. (315.5,132.8) .. controls (282.09,132.8) and (255,118.93) .. (255,101.8) -- cycle ;
\draw  [draw opacity=0] (318.51,188.31) .. controls (318.51,188.31) and (318.51,188.31) .. (318.51,188.31) .. controls (318.51,188.31) and (318.51,188.31) .. (318.51,188.31) .. controls (318.51,194) and (309.92,198.61) .. (299.31,198.61) .. controls (288.71,198.61) and (280.11,194) .. (280.11,188.31) .. controls (280.11,187.74) and (280.2,187.19) .. (280.36,186.65) -- (299.31,188.31) -- cycle ; \draw   (318.51,188.31) .. controls (318.51,188.31) and (318.51,188.31) .. (318.51,188.31) .. controls (318.51,188.31) and (318.51,188.31) .. (318.51,188.31) .. controls (318.51,194) and (309.92,198.61) .. (299.31,198.61) .. controls (288.71,198.61) and (280.11,194) .. (280.11,188.31) .. controls (280.11,187.74) and (280.2,187.19) .. (280.36,186.65) ;
\draw  [draw opacity=0] (282.06,192.85) .. controls (286.45,189.64) and (292.86,187.61) .. (300,187.61) .. controls (306.47,187.61) and (312.34,189.28) .. (316.65,191.98) -- (300,203.21) -- cycle ; \draw   (282.06,192.85) .. controls (286.45,189.64) and (292.86,187.61) .. (300,187.61) .. controls (306.47,187.61) and (312.34,189.28) .. (316.65,191.98) ;

\draw  [draw opacity=0] (366.51,189.31) .. controls (366.51,189.31) and (366.51,189.31) .. (366.51,189.31) .. controls (366.51,189.31) and (366.51,189.31) .. (366.51,189.31) .. controls (366.51,195) and (357.92,199.61) .. (347.31,199.61) .. controls (336.71,199.61) and (328.11,195) .. (328.11,189.31) .. controls (328.11,188.74) and (328.2,188.19) .. (328.36,187.65) -- (347.31,189.31) -- cycle ; \draw   (366.51,189.31) .. controls (366.51,189.31) and (366.51,189.31) .. (366.51,189.31) .. controls (366.51,189.31) and (366.51,189.31) .. (366.51,189.31) .. controls (366.51,195) and (357.92,199.61) .. (347.31,199.61) .. controls (336.71,199.61) and (328.11,195) .. (328.11,189.31) .. controls (328.11,188.74) and (328.2,188.19) .. (328.36,187.65) ;
\draw  [draw opacity=0] (330.06,193.85) .. controls (334.45,190.64) and (340.86,188.61) .. (348,188.61) .. controls (354.47,188.61) and (360.34,190.28) .. (364.65,192.98) -- (348,204.21) -- cycle ; \draw   (330.06,193.85) .. controls (334.45,190.64) and (340.86,188.61) .. (348,188.61) .. controls (354.47,188.61) and (360.34,190.28) .. (364.65,192.98) ;

\draw (239,89) node [anchor=north west][inner sep=0.75pt]   [align=left] {1};
\draw (237,182) node [anchor=north west][inner sep=0.75pt]   [align=left] {2};
\draw (261,183) node [anchor=north west][inner sep=0.75pt]   [align=left] {2};
\draw (262,90) node [anchor=north west][inner sep=0.75pt]   [align=left] {1};
\end{tikzpicture}
\end{minipage}
\begin{minipage}{0.3\linewidth}
\begin{tikzpicture}[x=0.75pt,y=0.75pt,yscale=-0.7,xscale=0.7]

\draw   (296,96.59) .. controls (296,81.59) and (327,74.59) .. (346,78.59) .. controls (365,82.59) and (373.46,84.1) .. (385,94.59) .. controls (396.54,105.09) and (405.45,129.15) .. (406,149.59) .. controls (406.55,170.04) and (401.96,182.26) .. (391,196.59) .. controls (380.04,210.93) and (379,211.59) .. (361,216.59) .. controls (343,221.59) and (297,211.59) .. (296,193.59) .. controls (295,175.59) and (304,182.59) .. (317,166.59) .. controls (330,150.59) and (331,131.59) .. (322,124.59) .. controls (313,117.59) and (296,111.59) .. (296,96.59) -- cycle ;
\draw  [draw opacity=0] (381.51,144.1) .. controls (381.51,144.1) and (381.51,144.1) .. (381.51,144.1) .. controls (381.51,144.1) and (381.51,144.1) .. (381.51,144.1) .. controls (381.51,149.79) and (372.92,154.4) .. (362.31,154.4) .. controls (351.71,154.4) and (343.11,149.79) .. (343.11,144.1) .. controls (343.11,143.53) and (343.2,142.98) .. (343.36,142.44) -- (362.31,144.1) -- cycle ; \draw   (381.51,144.1) .. controls (381.51,144.1) and (381.51,144.1) .. (381.51,144.1) .. controls (381.51,144.1) and (381.51,144.1) .. (381.51,144.1) .. controls (381.51,149.79) and (372.92,154.4) .. (362.31,154.4) .. controls (351.71,154.4) and (343.11,149.79) .. (343.11,144.1) .. controls (343.11,143.53) and (343.2,142.98) .. (343.36,142.44) ;
\draw  [draw opacity=0] (345.06,148.64) .. controls (349.45,145.43) and (355.86,143.41) .. (363,143.41) .. controls (369.47,143.41) and (375.34,145.07) .. (379.65,147.77) -- (363,159) -- cycle ; \draw   (345.06,148.64) .. controls (349.45,145.43) and (355.86,143.41) .. (363,143.41) .. controls (369.47,143.41) and (375.34,145.07) .. (379.65,147.77) ;

\draw   (172,190) .. controls (172,170.67) and (199.76,155) .. (234,155) .. controls (268.24,155) and (296,170.67) .. (296,190) .. controls (296,209.33) and (268.24,225) .. (234,225) .. controls (199.76,225) and (172,209.33) .. (172,190) -- cycle ;
\draw   (170,99.59) .. controls (170,83.58) and (198.21,70.59) .. (233,70.59) .. controls (267.79,70.59) and (296,83.58) .. (296,99.59) .. controls (296,115.61) and (267.79,128.59) .. (233,128.59) .. controls (198.21,128.59) and (170,115.61) .. (170,99.59) -- cycle ;
\draw  [draw opacity=0] (272.51,195.1) .. controls (272.51,195.1) and (272.51,195.1) .. (272.51,195.1) .. controls (272.51,195.1) and (272.51,195.1) .. (272.51,195.1) .. controls (272.51,200.79) and (263.92,205.4) .. (253.31,205.4) .. controls (242.71,205.4) and (234.11,200.79) .. (234.11,195.1) .. controls (234.11,194.53) and (234.2,193.98) .. (234.36,193.44) -- (253.31,195.1) -- cycle ; \draw   (272.51,195.1) .. controls (272.51,195.1) and (272.51,195.1) .. (272.51,195.1) .. controls (272.51,195.1) and (272.51,195.1) .. (272.51,195.1) .. controls (272.51,200.79) and (263.92,205.4) .. (253.31,205.4) .. controls (242.71,205.4) and (234.11,200.79) .. (234.11,195.1) .. controls (234.11,194.53) and (234.2,193.98) .. (234.36,193.44) ;
\draw  [draw opacity=0] (236.06,199.64) .. controls (240.45,196.43) and (246.86,194.41) .. (254,194.41) .. controls (260.47,194.41) and (266.34,196.07) .. (270.65,198.77) -- (254,210) -- cycle ; \draw   (236.06,199.64) .. controls (240.45,196.43) and (246.86,194.41) .. (254,194.41) .. controls (260.47,194.41) and (266.34,196.07) .. (270.65,198.77) ;

\draw  [draw opacity=0] (228.51,175.1) .. controls (228.51,175.1) and (228.51,175.1) .. (228.51,175.1) .. controls (228.51,175.1) and (228.51,175.1) .. (228.51,175.1) .. controls (228.51,180.79) and (219.92,185.4) .. (209.31,185.4) .. controls (198.71,185.4) and (190.11,180.79) .. (190.11,175.1) .. controls (190.11,174.53) and (190.2,173.98) .. (190.36,173.44) -- (209.31,175.1) -- cycle ; \draw   (228.51,175.1) .. controls (228.51,175.1) and (228.51,175.1) .. (228.51,175.1) .. controls (228.51,175.1) and (228.51,175.1) .. (228.51,175.1) .. controls (228.51,180.79) and (219.92,185.4) .. (209.31,185.4) .. controls (198.71,185.4) and (190.11,180.79) .. (190.11,175.1) .. controls (190.11,174.53) and (190.2,173.98) .. (190.36,173.44) ;
\draw  [draw opacity=0] (192.06,179.64) .. controls (196.45,176.43) and (202.86,174.41) .. (210,174.41) .. controls (216.47,174.41) and (222.34,176.07) .. (226.65,178.77) -- (210,190) -- cycle ; \draw   (192.06,179.64) .. controls (196.45,176.43) and (202.86,174.41) .. (210,174.41) .. controls (216.47,174.41) and (222.34,176.07) .. (226.65,178.77) ;

\draw  [draw opacity=0] (273.51,174.1) .. controls (273.51,174.1) and (273.51,174.1) .. (273.51,174.1) .. controls (273.51,174.1) and (273.51,174.1) .. (273.51,174.1) .. controls (273.51,179.79) and (264.92,184.4) .. (254.31,184.4) .. controls (243.71,184.4) and (235.11,179.79) .. (235.11,174.1) .. controls (235.11,173.53) and (235.2,172.98) .. (235.36,172.44) -- (254.31,174.1) -- cycle ; \draw   (273.51,174.1) .. controls (273.51,174.1) and (273.51,174.1) .. (273.51,174.1) .. controls (273.51,174.1) and (273.51,174.1) .. (273.51,174.1) .. controls (273.51,179.79) and (264.92,184.4) .. (254.31,184.4) .. controls (243.71,184.4) and (235.11,179.79) .. (235.11,174.1) .. controls (235.11,173.53) and (235.2,172.98) .. (235.36,172.44) ;
\draw  [draw opacity=0] (237.06,178.64) .. controls (241.45,175.43) and (247.86,173.41) .. (255,173.41) .. controls (261.47,173.41) and (267.34,175.07) .. (271.65,177.77) -- (255,189) -- cycle ; \draw   (237.06,178.64) .. controls (241.45,175.43) and (247.86,173.41) .. (255,173.41) .. controls (261.47,173.41) and (267.34,175.07) .. (271.65,177.77) ;

\draw (302,91.79) node [anchor=north west][inner sep=0.75pt]   [align=left] {1};
\draw (302,182.79) node [anchor=north west][inner sep=0.75pt]   [align=left] {2};
\draw (275,90.79) node [anchor=north west][inner sep=0.75pt]   [align=left] {1};
\draw (279,182.79) node [anchor=north west][inner sep=0.75pt]   [align=left] {2};
\end{tikzpicture}
\end{minipage}
\begin{minipage}{0.3\linewidth}
\begin{tikzpicture}[x=0.75pt,y=0.75pt,yscale=-0.7,xscale=0.7]

\draw   (54,57.8) .. controls (68,46.8) and (89,44.8) .. (112,44.8) .. controls (135,44.8) and (153,66.8) .. (137,80.8) .. controls (121,94.8) and (112,92.8) .. (112,109.8) .. controls (112,126.8) and (127,136.8) .. (138,146.8) .. controls (149,156.8) and (142,178.8) .. (117,186.8) .. controls (92,194.8) and (58,180.61) .. (46,163.8) .. controls (34,147) and (26.14,138.07) .. (26,116.8) .. controls (25.86,95.54) and (40,68.8) .. (54,57.8) -- cycle ;
\draw   (143,63.8) .. controls (143,48.8) and (174,41.8) .. (193,45.8) .. controls (212,49.8) and (220.46,51.31) .. (232,61.8) .. controls (243.54,72.3) and (252.45,96.36) .. (253,116.8) .. controls (253.55,137.25) and (248.96,149.47) .. (238,163.8) .. controls (227.04,178.14) and (226,178.8) .. (208,183.8) .. controls (190,188.8) and (144,178.8) .. (143,160.8) .. controls (142,142.8) and (151,149.8) .. (164,133.8) .. controls (177,117.8) and (178,98.8) .. (169,91.8) .. controls (160,84.8) and (143,78.8) .. (143,63.8) -- cycle ;
\draw  [draw opacity=0] (95.51,94.31) .. controls (95.51,94.31) and (95.51,94.31) .. (95.51,94.31) .. controls (95.51,94.31) and (95.51,94.31) .. (95.51,94.31) .. controls (95.51,100) and (86.92,104.61) .. (76.31,104.61) .. controls (65.71,104.61) and (57.11,100) .. (57.11,94.31) .. controls (57.11,93.74) and (57.2,93.19) .. (57.36,92.65) -- (76.31,94.31) -- cycle ; \draw   (95.51,94.31) .. controls (95.51,94.31) and (95.51,94.31) .. (95.51,94.31) .. controls (95.51,94.31) and (95.51,94.31) .. (95.51,94.31) .. controls (95.51,100) and (86.92,104.61) .. (76.31,104.61) .. controls (65.71,104.61) and (57.11,100) .. (57.11,94.31) .. controls (57.11,93.74) and (57.2,93.19) .. (57.36,92.65) ;
\draw  [draw opacity=0] (59.06,98.85) .. controls (63.45,95.64) and (69.86,93.61) .. (77,93.61) .. controls (83.47,93.61) and (89.34,95.28) .. (93.65,97.98) -- (77,109.21) -- cycle ; \draw   (59.06,98.85) .. controls (63.45,95.64) and (69.86,93.61) .. (77,93.61) .. controls (83.47,93.61) and (89.34,95.28) .. (93.65,97.98) ;

\draw  [draw opacity=0] (95.51,136.31) .. controls (95.51,136.31) and (95.51,136.31) .. (95.51,136.31) .. controls (95.51,136.31) and (95.51,136.31) .. (95.51,136.31) .. controls (95.51,142) and (86.92,146.61) .. (76.31,146.61) .. controls (65.71,146.61) and (57.11,142) .. (57.11,136.31) .. controls (57.11,135.74) and (57.2,135.19) .. (57.36,134.65) -- (76.31,136.31) -- cycle ; \draw   (95.51,136.31) .. controls (95.51,136.31) and (95.51,136.31) .. (95.51,136.31) .. controls (95.51,136.31) and (95.51,136.31) .. (95.51,136.31) .. controls (95.51,142) and (86.92,146.61) .. (76.31,146.61) .. controls (65.71,146.61) and (57.11,142) .. (57.11,136.31) .. controls (57.11,135.74) and (57.2,135.19) .. (57.36,134.65) ;
\draw  [draw opacity=0] (59.06,140.85) .. controls (63.45,137.64) and (69.86,135.61) .. (77,135.61) .. controls (83.47,135.61) and (89.34,137.28) .. (93.65,139.98) -- (77,151.21) -- cycle ; \draw   (59.06,140.85) .. controls (63.45,137.64) and (69.86,135.61) .. (77,135.61) .. controls (83.47,135.61) and (89.34,137.28) .. (93.65,139.98) ;

\draw  [draw opacity=0] (228.51,111.31) .. controls (228.51,111.31) and (228.51,111.31) .. (228.51,111.31) .. controls (228.51,111.31) and (228.51,111.31) .. (228.51,111.31) .. controls (228.51,117) and (219.92,121.61) .. (209.31,121.61) .. controls (198.71,121.61) and (190.11,117) .. (190.11,111.31) .. controls (190.11,110.74) and (190.2,110.19) .. (190.36,109.65) -- (209.31,111.31) -- cycle ; \draw   (228.51,111.31) .. controls (228.51,111.31) and (228.51,111.31) .. (228.51,111.31) .. controls (228.51,111.31) and (228.51,111.31) .. (228.51,111.31) .. controls (228.51,117) and (219.92,121.61) .. (209.31,121.61) .. controls (198.71,121.61) and (190.11,117) .. (190.11,111.31) .. controls (190.11,110.74) and (190.2,110.19) .. (190.36,109.65) ;
\draw  [draw opacity=0] (192.06,115.85) .. controls (196.45,112.64) and (202.86,110.61) .. (210,110.61) .. controls (216.47,110.61) and (222.34,112.28) .. (226.65,114.98) -- (210,126.21) -- cycle ; \draw   (192.06,115.85) .. controls (196.45,112.64) and (202.86,110.61) .. (210,110.61) .. controls (216.47,110.61) and (222.34,112.28) .. (226.65,114.98) ;

\draw (127,59) node [anchor=north west][inner sep=0.75pt]   [align=left] {1};
\draw (125,152) node [anchor=north west][inner sep=0.75pt]   [align=left] {2};
\draw (149,59) node [anchor=north west][inner sep=0.75pt]   [align=left] {1};
\draw (149,153) node [anchor=north west][inner sep=0.75pt]   [align=left] {2};
\end{tikzpicture}
\end{minipage}
\end{center}

\vspace{0.25cm}

The subsets of $\sHbar_{3,4}$ parameterizing $\alpha: C \rightarrow P$ in which $C$ has each of these geometries are the three codimension-$1$ boundary strata lying over $D_7 \subseteq \Mbar_{0,12}/S_{12}$.  Note that the first of these strata is the image of a gluing map
\[\sH_{3,2}((2,1)) \times \sH_{(2,1),(2,0)}((2),(1)) \rightarrow \sHbar_{3,4},\]
and similarly for the other two.
\end{example}

Generalizing this example, we characterize the codimension-$1$ boundary strata of $\sHbar_{3,g}$ in the following lemma.

\begin{lemma}\label{lem:BDs}
The codimension-$1$ boundary strata in $\sHbar_{3,g}$ are the images of gluing maps from products of two moduli spaces (which may be the same) from among the following list:
\begin{itemize}
\item $\sH_{3,g'}((3))$ for $0 \leq g' \leq g$,
\item $\sH_{3,g'}((2,1))$ for $0 \leq g' \leq g$,
\item $\sH_{3,g'}((1,1,1))$ for $0 \leq g' \leq g$,
\item $\sH_{(2,1),(g',0)}((2),(1))$ for $1 \leq g' \leq g$,
\item $\sH_{(2,1),(g',0)}((1,1),(1))$ for $0 \leq g' \leq g$.
\end{itemize}
Furthermore, these gluing maps are finite group quotients onto their images.
\end{lemma}
\begin{proof}
Let $\mathcal{S} \subseteq \sHbar_{3,g}$ be a codimension-$1$ boundary stratum lying over a stratum $D_j \subseteq \Mbar_{0,b}/S_b$, and let $\alpha: C \rightarrow P$ be an element of $\mathcal{S}$.  Then $P$ has two components $P_1$ and $P_2$ (where $P_1$ has $j$ smooth branch points and $P_2$ has $b-j$), and the ramification profile of $\alpha$ over the node at which $P_1$ and $P_2$ meet is either $(3)$, $(2,1)$, or $(1,1,1)$.  In the first case, $C$ must consist of two components meeting at a single node:

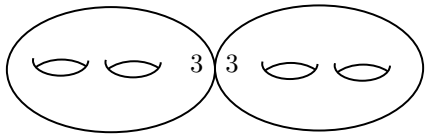
\tikzset{every picture/.style={line width=0.75pt}} 

\begin{center}
\begin{tikzpicture}[x=0.75pt,y=0.75pt,yscale=-0.85,xscale=0.85]

\draw   (100,143.15) .. controls (100,122.63) and (127.65,106) .. (161.75,106) .. controls (195.85,106) and (223.5,122.63) .. (223.5,143.15) .. controls (223.5,163.67) and (195.85,180.3) .. (161.75,180.3) .. controls (127.65,180.3) and (100,163.67) .. (100,143.15) -- cycle ;
\draw  [draw opacity=0] (148.31,137.76) .. controls (148.31,137.76) and (148.31,137.76) .. (148.31,137.76) .. controls (148.31,142.71) and (140.97,146.73) .. (131.91,146.73) .. controls (122.86,146.73) and (115.51,142.71) .. (115.51,137.76) .. controls (115.51,137.28) and (115.58,136.81) .. (115.72,136.34) -- (131.91,137.76) -- cycle ; \draw   (148.31,137.76) .. controls (148.31,137.76) and (148.31,137.76) .. (148.31,137.76) .. controls (148.31,142.71) and (140.97,146.73) .. (131.91,146.73) .. controls (122.86,146.73) and (115.51,142.71) .. (115.51,137.76) .. controls (115.51,137.28) and (115.58,136.81) .. (115.72,136.34) ;
\draw  [draw opacity=0] (117.05,141.81) .. controls (120.81,138.96) and (126.34,137.16) .. (132.5,137.16) .. controls (138.09,137.16) and (143.16,138.64) .. (146.86,141.05) -- (132.5,150.73) -- cycle ; \draw   (117.05,141.81) .. controls (120.81,138.96) and (126.34,137.16) .. (132.5,137.16) .. controls (138.09,137.16) and (143.16,138.64) .. (146.86,141.05) ;

\draw  [draw opacity=0] (191.31,138.76) .. controls (191.31,138.76) and (191.31,138.76) .. (191.31,138.76) .. controls (191.31,143.71) and (183.97,147.73) .. (174.91,147.73) .. controls (165.86,147.73) and (158.51,143.71) .. (158.51,138.76) .. controls (158.51,138.28) and (158.58,137.81) .. (158.72,137.34) -- (174.91,138.76) -- cycle ; \draw   (191.31,138.76) .. controls (191.31,138.76) and (191.31,138.76) .. (191.31,138.76) .. controls (191.31,143.71) and (183.97,147.73) .. (174.91,147.73) .. controls (165.86,147.73) and (158.51,143.71) .. (158.51,138.76) .. controls (158.51,138.28) and (158.58,137.81) .. (158.72,137.34) ;
\draw  [draw opacity=0] (160.05,142.81) .. controls (163.81,139.96) and (169.34,138.16) .. (175.5,138.16) .. controls (181.09,138.16) and (186.16,139.64) .. (189.86,142.05) -- (175.5,151.73) -- cycle ; \draw   (160.05,142.81) .. controls (163.81,139.96) and (169.34,138.16) .. (175.5,138.16) .. controls (181.09,138.16) and (186.16,139.64) .. (189.86,142.05) ;

\draw   (223.5,142.15) .. controls (223.5,121.63) and (251.15,105) .. (285.25,105) .. controls (319.35,105) and (347,121.63) .. (347,142.15) .. controls (347,162.67) and (319.35,179.3) .. (285.25,179.3) .. controls (251.15,179.3) and (223.5,162.67) .. (223.5,142.15) -- cycle ;
\draw  [draw opacity=0] (284.31,139.76) .. controls (284.31,139.76) and (284.31,139.76) .. (284.31,139.76) .. controls (284.31,144.71) and (276.97,148.73) .. (267.91,148.73) .. controls (258.86,148.73) and (251.51,144.71) .. (251.51,139.76) .. controls (251.51,139.28) and (251.58,138.81) .. (251.72,138.34) -- (267.91,139.76) -- cycle ; \draw   (284.31,139.76) .. controls (284.31,139.76) and (284.31,139.76) .. (284.31,139.76) .. controls (284.31,144.71) and (276.97,148.73) .. (267.91,148.73) .. controls (258.86,148.73) and (251.51,144.71) .. (251.51,139.76) .. controls (251.51,139.28) and (251.58,138.81) .. (251.72,138.34) ;
\draw  [draw opacity=0] (253.05,143.81) .. controls (256.81,140.96) and (262.34,139.16) .. (268.5,139.16) .. controls (274.09,139.16) and (279.16,140.64) .. (282.86,143.05) -- (268.5,152.73) -- cycle ; \draw   (253.05,143.81) .. controls (256.81,140.96) and (262.34,139.16) .. (268.5,139.16) .. controls (274.09,139.16) and (279.16,140.64) .. (282.86,143.05) ;

\draw  [draw opacity=0] (327.31,140.76) .. controls (327.31,140.76) and (327.31,140.76) .. (327.31,140.76) .. controls (327.31,145.71) and (319.97,149.73) .. (310.91,149.73) .. controls (301.86,149.73) and (294.51,145.71) .. (294.51,140.76) .. controls (294.51,140.28) and (294.58,139.81) .. (294.72,139.34) -- (310.91,140.76) -- cycle ; \draw   (327.31,140.76) .. controls (327.31,140.76) and (327.31,140.76) .. (327.31,140.76) .. controls (327.31,145.71) and (319.97,149.73) .. (310.91,149.73) .. controls (301.86,149.73) and (294.51,145.71) .. (294.51,140.76) .. controls (294.51,140.28) and (294.58,139.81) .. (294.72,139.34) ;
\draw  [draw opacity=0] (296.05,144.81) .. controls (299.81,141.96) and (305.34,140.16) .. (311.5,140.16) .. controls (317.09,140.16) and (322.16,141.64) .. (325.86,144.05) -- (311.5,153.73) -- cycle ; \draw   (296.05,144.81) .. controls (299.81,141.96) and (305.34,140.16) .. (311.5,140.16) .. controls (317.09,140.16) and (322.16,141.64) .. (325.86,144.05) ;

\draw (207,132.79) node [anchor=north west][inner sep=0.75pt]   [align=left] {3};
\draw (228,132.79) node [anchor=north west][inner sep=0.75pt]   [align=left] {3};
\end{tikzpicture}
\end{center}

\noindent Note that the genera $g_1$ and $g_2$ of the two components of $C$ are determined from $g$ and $j$ by the Riemann--Hurwitz formula.  (In particular, $g_1 = (j-2)/2$, so the above picture illustrates the case $g=4$ and $j=6$.)  Therefore, in this case, $\mathcal{S}$ is the image of a gluing map
\[\sH_{3,g_1}((3)) \times \sH_{3,g_2}((3)) \rightarrow \sHbar_{3,g}.\]
If $g_1 \neq g_2$, then this gluing map is injective, that is, it is fully faithful as a map of algebraic stacks.  On the other hand, if $g_1 = g_2$, then the gluing map is a $\ZZ_2$-quotient onto its image, accounting for the fact that the product chooses an ordering of the two components.  (The two choices of ordering are equal if the two components happen to be isomorphic, but in this case, the image in $\sHbar_{3,g}$ has an extra $\ZZ_2$ automorphism group not present in the domain of the gluing map, so the map is still a $\ZZ_2$-quotient in the orbifold sense.)

If, on the other hand, the ramification profile of $\alpha$ over the node of $P$ is $(2,1)$, then, as in Example~\ref{ex:BD}, it can be the case that one (but at most one) of $\alpha^{-1}(P_1)$ or $\alpha^{-1}(P_2)$ is disconnected, and the possible geometries of $C$ are as illustrated in that example.  If $\alpha^{-1}(P_1)$ is connected, then its genus is determined by Riemann--Hurwitz: $g_1= (j-3)/2$.  If $\alpha^{-1}(P_2)$ is connected, then its genus is $g_2 =(b-j-3)/2$. If $\alpha^{-1}(P_1)$ is disconnected, then it consists of one component mapping to $P_1$ with degree $1$ (and therefore necessarily of genus zero) and another of genus $g_1=(j-1)/2$ mapping with degree $2$; note, in this case, that we cannot have $g_1 = 0$, since $j\geq2$.  Similar reasoning applies if $\alpha^{-1}(P_2)$ is disconnected: it consists of one component mapping to $P_2$ with degree $1$ and  genus zero and another of genus $g_2=(b-j-1)/2\geq1$ mapping with degree $2$. For each of the possible cases---that $\alpha^{-1}(P_1)$ and $\alpha^{-1}(P_2)$ are both connected, $\alpha^{-1}(P_1)$ is connected but $\alpha^{-1}(P_2)$ is disconnected, and $\alpha^{-1}(P_2)$ is connected but $\alpha^{-1}(P_1)$ is disconnected---the stratum $\mathcal{S}$ is the image of a corresponding gluing map:
\[\sH_{3,g_1}((2,1)) \times \sH_{3, g_2}((2,1)) \rightarrow \sHbar_{3,g},\]
which is injective for $g_1 \neq g_2$ and a $\ZZ_2$-quotient onto its image for $g_1 = g_2$, or one of
\[\sH_{3,g_1}((2,1)) \times \sH_{(2,1), (g_2, 0)}((2), (1)) \rightarrow \sHbar_{3,g},\]
\[\sH_{(2,1),(g_1,0)}((2),(1)) \times \sH_{3,g_2}((2, 1)) \rightarrow \sHbar_{3,g},\]
which are injective and have $g_1,g_2\geq1$.

Finally, if the ramification profile of $\alpha$ over the node of $P$ is $(1,1,1)$, then it can be the case that one or both of $\alpha^{-1}(P_1)$ and $\alpha^{-1}(P_2)$ are disconnected, yielding the following possible geometries of $C$ (illustrated in the case where $g=4$ and $j=6$):

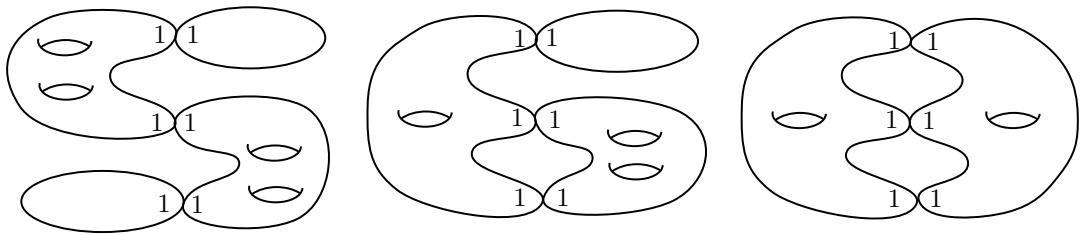
\tikzset{every picture/.style={line width=0.75pt}} 

\begin{center}
\begin{minipage}{0.3\linewidth}
\begin{tikzpicture}[x=0.75pt,y=0.75pt,yscale=-0.7,xscale=0.7]

\draw   (229,163.8) .. controls (221,140.8) and (302,133.3) .. (325,150.3) .. controls (348,167.3) and (343,216.3) .. (317,229.3) .. controls (291,242.3) and (228,235.8) .. (235,217.8) .. controls (242,199.8) and (275,202.3) .. (275,189.3) .. controls (275,176.3) and (237,186.8) .. (229,163.8) -- cycle ;
\draw   (138,84.3) .. controls (164,73.3) and (238,78.5) .. (229,98.8) .. controls (220,119.11) and (182,109.5) .. (182,126.3) .. controls (182,143.11) and (224,141.11) .. (229,158.8) .. controls (234,176.5) and (137,178.3) .. (117,148.3) .. controls (97,118.3) and (112,95.3) .. (138,84.3) -- cycle ;
\draw   (118,216.8) .. controls (118,204.63) and (144.19,194.76) .. (176.5,194.76) .. controls (208.81,194.76) and (235,204.63) .. (235,216.8) .. controls (235,228.98) and (208.81,238.85) .. (176.5,238.85) .. controls (144.19,238.85) and (118,228.98) .. (118,216.8) -- cycle ;
\draw   (229,98.8) .. controls (229,86.63) and (253.18,76.76) .. (283,76.76) .. controls (312.82,76.76) and (337,86.63) .. (337,98.8) .. controls (337,110.98) and (312.82,120.85) .. (283,120.85) .. controls (253.18,120.85) and (229,110.98) .. (229,98.8) -- cycle ;
\draw  [draw opacity=0] (319.51,177.1) .. controls (319.51,177.1) and (319.51,177.1) .. (319.51,177.1) .. controls (319.51,177.1) and (319.51,177.1) .. (319.51,177.1) .. controls (319.51,182.79) and (310.92,187.4) .. (300.31,187.4) .. controls (289.71,187.4) and (281.11,182.79) .. (281.11,177.1) .. controls (281.11,176.53) and (281.2,175.98) .. (281.36,175.44) -- (300.31,177.1) -- cycle ; \draw   (319.51,177.1) .. controls (319.51,177.1) and (319.51,177.1) .. (319.51,177.1) .. controls (319.51,177.1) and (319.51,177.1) .. (319.51,177.1) .. controls (319.51,182.79) and (310.92,187.4) .. (300.31,187.4) .. controls (289.71,187.4) and (281.11,182.79) .. (281.11,177.1) .. controls (281.11,176.53) and (281.2,175.98) .. (281.36,175.44) ;
\draw  [draw opacity=0] (283.06,181.64) .. controls (287.45,178.43) and (293.86,176.41) .. (301,176.41) .. controls (307.47,176.41) and (313.34,178.07) .. (317.65,180.77) -- (301,192) -- cycle ; \draw   (283.06,181.64) .. controls (287.45,178.43) and (293.86,176.41) .. (301,176.41) .. controls (307.47,176.41) and (313.34,178.07) .. (317.65,180.77) ;

\draw  [draw opacity=0] (168.51,101.1) .. controls (168.51,101.1) and (168.51,101.1) .. (168.51,101.1) .. controls (168.51,101.1) and (168.51,101.1) .. (168.51,101.1) .. controls (168.51,106.79) and (159.92,111.4) .. (149.31,111.4) .. controls (138.71,111.4) and (130.11,106.79) .. (130.11,101.1) .. controls (130.11,100.53) and (130.2,99.98) .. (130.36,99.44) -- (149.31,101.1) -- cycle ; \draw   (168.51,101.1) .. controls (168.51,101.1) and (168.51,101.1) .. (168.51,101.1) .. controls (168.51,101.1) and (168.51,101.1) .. (168.51,101.1) .. controls (168.51,106.79) and (159.92,111.4) .. (149.31,111.4) .. controls (138.71,111.4) and (130.11,106.79) .. (130.11,101.1) .. controls (130.11,100.53) and (130.2,99.98) .. (130.36,99.44) ;
\draw  [draw opacity=0] (132.06,105.64) .. controls (136.45,102.43) and (142.86,100.41) .. (150,100.41) .. controls (156.47,100.41) and (162.34,102.07) .. (166.65,104.77) -- (150,116) -- cycle ; \draw   (132.06,105.64) .. controls (136.45,102.43) and (142.86,100.41) .. (150,100.41) .. controls (156.47,100.41) and (162.34,102.07) .. (166.65,104.77) ;

\draw  [draw opacity=0] (320.51,207.1) .. controls (320.51,207.1) and (320.51,207.1) .. (320.51,207.1) .. controls (320.51,207.1) and (320.51,207.1) .. (320.51,207.1) .. controls (320.51,212.79) and (311.92,217.4) .. (301.31,217.4) .. controls (290.71,217.4) and (282.11,212.79) .. (282.11,207.1) .. controls (282.11,206.53) and (282.2,205.98) .. (282.36,205.44) -- (301.31,207.1) -- cycle ; \draw   (320.51,207.1) .. controls (320.51,207.1) and (320.51,207.1) .. (320.51,207.1) .. controls (320.51,207.1) and (320.51,207.1) .. (320.51,207.1) .. controls (320.51,212.79) and (311.92,217.4) .. (301.31,217.4) .. controls (290.71,217.4) and (282.11,212.79) .. (282.11,207.1) .. controls (282.11,206.53) and (282.2,205.98) .. (282.36,205.44) ;
\draw  [draw opacity=0] (284.06,211.64) .. controls (288.45,208.43) and (294.86,206.41) .. (302,206.41) .. controls (308.47,206.41) and (314.34,208.07) .. (318.65,210.77) -- (302,222) -- cycle ; \draw   (284.06,211.64) .. controls (288.45,208.43) and (294.86,206.41) .. (302,206.41) .. controls (308.47,206.41) and (314.34,208.07) .. (318.65,210.77) ;

\draw  [draw opacity=0] (169.51,133.1) .. controls (169.51,133.1) and (169.51,133.1) .. (169.51,133.1) .. controls (169.51,133.1) and (169.51,133.1) .. (169.51,133.1) .. controls (169.51,138.79) and (160.92,143.4) .. (150.31,143.4) .. controls (139.71,143.4) and (131.11,138.79) .. (131.11,133.1) .. controls (131.11,132.53) and (131.2,131.98) .. (131.36,131.44) -- (150.31,133.1) -- cycle ; \draw   (169.51,133.1) .. controls (169.51,133.1) and (169.51,133.1) .. (169.51,133.1) .. controls (169.51,133.1) and (169.51,133.1) .. (169.51,133.1) .. controls (169.51,138.79) and (160.92,143.4) .. (150.31,143.4) .. controls (139.71,143.4) and (131.11,138.79) .. (131.11,133.1) .. controls (131.11,132.53) and (131.2,131.98) .. (131.36,131.44) ;
\draw  [draw opacity=0] (133.06,137.64) .. controls (137.45,134.43) and (143.86,132.41) .. (151,132.41) .. controls (157.47,132.41) and (163.34,134.07) .. (167.65,136.77) -- (151,148) -- cycle ; \draw   (133.06,137.64) .. controls (137.45,134.43) and (143.86,132.41) .. (151,132.41) .. controls (157.47,132.41) and (163.34,134.07) .. (167.65,136.77) ;

\draw (211,87.79) node [anchor=north west][inner sep=0.75pt]   [align=left] {1};
\draw (235,87.79) node [anchor=north west][inner sep=0.75pt]   [align=left] {1};
\draw (209,151.79) node [anchor=north west][inner sep=0.75pt]   [align=left] {1};
\draw (233,151.79) node [anchor=north west][inner sep=0.75pt]   [align=left] {1};
\draw (214,209.79) node [anchor=north west][inner sep=0.75pt]   [align=left] {1};
\draw (238,210.79) node [anchor=north west][inner sep=0.75pt]   [align=left] {1};
\end{tikzpicture}
\end{minipage}
\begin{minipage}{0.3\linewidth}
\begin{tikzpicture}[x=0.75pt,y=0.75pt,yscale=-0.7,xscale=0.7]

\draw   (125,79.8) .. controls (144,67.8) and (200,65.8) .. (209,83.8) .. controls (218,101.8) and (161,91.8) .. (160,111.8) .. controls (159,131.8) and (209,129.8) .. (209,146.8) .. controls (209,163.8) and (164,153.8) .. (163,169.8) .. controls (162,185.8) and (227,190.8) .. (212,207.8) .. controls (197,224.8) and (128,210.8) .. (109,195.8) .. controls (90,180.8) and (87,166.8) .. (88,133.8) .. controls (89,100.8) and (106,91.8) .. (125,79.8) -- cycle ;
\draw  [draw opacity=0] (148.51,140.1) .. controls (148.51,140.1) and (148.51,140.1) .. (148.51,140.1) .. controls (148.51,140.1) and (148.51,140.1) .. (148.51,140.1) .. controls (148.51,145.79) and (139.92,150.4) .. (129.31,150.4) .. controls (118.71,150.4) and (110.11,145.79) .. (110.11,140.1) .. controls (110.11,139.53) and (110.2,138.98) .. (110.36,138.44) -- (129.31,140.1) -- cycle ; \draw   (148.51,140.1) .. controls (148.51,140.1) and (148.51,140.1) .. (148.51,140.1) .. controls (148.51,140.1) and (148.51,140.1) .. (148.51,140.1) .. controls (148.51,145.79) and (139.92,150.4) .. (129.31,150.4) .. controls (118.71,150.4) and (110.11,145.79) .. (110.11,140.1) .. controls (110.11,139.53) and (110.2,138.98) .. (110.36,138.44) ;
\draw  [draw opacity=0] (112.06,144.64) .. controls (116.45,141.43) and (122.86,139.41) .. (130,139.41) .. controls (136.47,139.41) and (142.34,141.07) .. (146.65,143.77) -- (130,155) -- cycle ; \draw   (112.06,144.64) .. controls (116.45,141.43) and (122.86,139.41) .. (130,139.41) .. controls (136.47,139.41) and (142.34,141.07) .. (146.65,143.77) ;

\draw   (209,146.8) .. controls (201,123.8) and (288,124.8) .. (315,140.8) .. controls (342,156.8) and (334,190.8) .. (308,203.8) .. controls (282,216.8) and (208,218.8) .. (215,200.8) .. controls (222,182.8) and (250,185.8) .. (250,172.8) .. controls (250,159.8) and (217,169.8) .. (209,146.8) -- cycle ;
\draw   (209,88.76) .. controls (209,76.58) and (235.19,66.71) .. (267.5,66.71) .. controls (299.81,66.71) and (326,76.58) .. (326,88.76) .. controls (326,100.93) and (299.81,110.8) .. (267.5,110.8) .. controls (235.19,110.8) and (209,100.93) .. (209,88.76) -- cycle ;
\draw  [draw opacity=0] (299.51,154.1) .. controls (299.51,154.1) and (299.51,154.1) .. (299.51,154.1) .. controls (299.51,154.1) and (299.51,154.1) .. (299.51,154.1) .. controls (299.51,159.79) and (290.92,164.4) .. (280.31,164.4) .. controls (269.71,164.4) and (261.11,159.79) .. (261.11,154.1) .. controls (261.11,153.53) and (261.2,152.98) .. (261.36,152.44) -- (280.31,154.1) -- cycle ; \draw   (299.51,154.1) .. controls (299.51,154.1) and (299.51,154.1) .. (299.51,154.1) .. controls (299.51,154.1) and (299.51,154.1) .. (299.51,154.1) .. controls (299.51,159.79) and (290.92,164.4) .. (280.31,164.4) .. controls (269.71,164.4) and (261.11,159.79) .. (261.11,154.1) .. controls (261.11,153.53) and (261.2,152.98) .. (261.36,152.44) ;
\draw  [draw opacity=0] (263.06,158.64) .. controls (267.45,155.43) and (273.86,153.41) .. (281,153.41) .. controls (287.47,153.41) and (293.34,155.07) .. (297.65,157.77) -- (281,169) -- cycle ; \draw   (263.06,158.64) .. controls (267.45,155.43) and (273.86,153.41) .. (281,153.41) .. controls (287.47,153.41) and (293.34,155.07) .. (297.65,157.77) ;

\draw  [draw opacity=0] (300.51,178.1) .. controls (300.51,178.1) and (300.51,178.1) .. (300.51,178.1) .. controls (300.51,178.1) and (300.51,178.1) .. (300.51,178.1) .. controls (300.51,183.79) and (291.92,188.4) .. (281.31,188.4) .. controls (270.71,188.4) and (262.11,183.79) .. (262.11,178.1) .. controls (262.11,177.53) and (262.2,176.98) .. (262.36,176.44) -- (281.31,178.1) -- cycle ; \draw   (300.51,178.1) .. controls (300.51,178.1) and (300.51,178.1) .. (300.51,178.1) .. controls (300.51,178.1) and (300.51,178.1) .. (300.51,178.1) .. controls (300.51,183.79) and (291.92,188.4) .. (281.31,188.4) .. controls (270.71,188.4) and (262.11,183.79) .. (262.11,178.1) .. controls (262.11,177.53) and (262.2,176.98) .. (262.36,176.44) ;
\draw  [draw opacity=0] (264.06,182.64) .. controls (268.45,179.43) and (274.86,177.41) .. (282,177.41) .. controls (288.47,177.41) and (294.34,179.07) .. (298.65,181.77) -- (282,193) -- cycle ; \draw   (264.06,182.64) .. controls (268.45,179.43) and (274.86,177.41) .. (282,177.41) .. controls (288.47,177.41) and (294.34,179.07) .. (298.65,181.77) ;

\draw (191,78.79) node [anchor=north west][inner sep=0.75pt]   [align=left] {1};
\draw (189,135.79) node [anchor=north west][inner sep=0.75pt]   [align=left] {1};
\draw (191,192.79) node [anchor=north west][inner sep=0.75pt]   [align=left] {1};
\draw (214,77.79) node [anchor=north west][inner sep=0.75pt]   [align=left] {1};
\draw (216,136.79) node [anchor=north west][inner sep=0.75pt]   [align=left] {1};
\draw (222,192.79) node [anchor=north west][inner sep=0.75pt]   [align=left] {1};

\end{tikzpicture}
\end{minipage}
\begin{minipage}{0.3\linewidth}
\begin{tikzpicture}[x=0.75pt,y=0.75pt,yscale=-0.7,xscale=0.7]

\draw   (105,59.8) .. controls (124,47.8) and (180,45.8) .. (189,63.8) .. controls (198,81.8) and (141,71.8) .. (140,91.8) .. controls (139,111.8) and (189,109.8) .. (189,126.8) .. controls (189,143.8) and (144,133.8) .. (143,149.8) .. controls (142,165.8) and (207,170.8) .. (192,187.8) .. controls (177,204.8) and (108,190.8) .. (89,175.8) .. controls (70,160.8) and (67,146.8) .. (68,113.8) .. controls (69,80.8) and (86,71.8) .. (105,59.8) -- cycle ;
\draw   (227,95.8) .. controls (227,80.8) and (178,76.8) .. (192,64.8) .. controls (206,52.8) and (244,43.8) .. (274,61.8) .. controls (304,79.8) and (309.45,100.36) .. (310,120.8) .. controls (310.55,141.25) and (307.96,149.47) .. (297,163.8) .. controls (286.04,178.14) and (278,185.8) .. (260,190.8) .. controls (242,195.8) and (203,198.8) .. (196,183.8) .. controls (189,168.8) and (231,171.8) .. (231,155.8) .. controls (231,139.8) and (190,141.8) .. (189,126.8) .. controls (188,111.8) and (227,110.8) .. (227,95.8) -- cycle ;
\draw  [draw opacity=0] (128.51,120.1) .. controls (128.51,120.1) and (128.51,120.1) .. (128.51,120.1) .. controls (128.51,120.1) and (128.51,120.1) .. (128.51,120.1) .. controls (128.51,125.79) and (119.92,130.4) .. (109.31,130.4) .. controls (98.71,130.4) and (90.11,125.79) .. (90.11,120.1) .. controls (90.11,119.53) and (90.2,118.98) .. (90.36,118.44) -- (109.31,120.1) -- cycle ; \draw   (128.51,120.1) .. controls (128.51,120.1) and (128.51,120.1) .. (128.51,120.1) .. controls (128.51,120.1) and (128.51,120.1) .. (128.51,120.1) .. controls (128.51,125.79) and (119.92,130.4) .. (109.31,130.4) .. controls (98.71,130.4) and (90.11,125.79) .. (90.11,120.1) .. controls (90.11,119.53) and (90.2,118.98) .. (90.36,118.44) ;
\draw  [draw opacity=0] (92.06,124.64) .. controls (96.45,121.43) and (102.86,119.41) .. (110,119.41) .. controls (116.47,119.41) and (122.34,121.07) .. (126.65,123.77) -- (110,135) -- cycle ; \draw   (92.06,124.64) .. controls (96.45,121.43) and (102.86,119.41) .. (110,119.41) .. controls (116.47,119.41) and (122.34,121.07) .. (126.65,123.77) ;

\draw  [draw opacity=0] (282.51,120.1) .. controls (282.51,120.1) and (282.51,120.1) .. (282.51,120.1) .. controls (282.51,120.1) and (282.51,120.1) .. (282.51,120.1) .. controls (282.51,125.79) and (273.92,130.4) .. (263.31,130.4) .. controls (252.71,130.4) and (244.11,125.79) .. (244.11,120.1) .. controls (244.11,119.53) and (244.2,118.98) .. (244.36,118.44) -- (263.31,120.1) -- cycle ; \draw   (282.51,120.1) .. controls (282.51,120.1) and (282.51,120.1) .. (282.51,120.1) .. controls (282.51,120.1) and (282.51,120.1) .. (282.51,120.1) .. controls (282.51,125.79) and (273.92,130.4) .. (263.31,130.4) .. controls (252.71,130.4) and (244.11,125.79) .. (244.11,120.1) .. controls (244.11,119.53) and (244.2,118.98) .. (244.36,118.44) ;
\draw  [draw opacity=0] (246.06,124.64) .. controls (250.45,121.43) and (256.86,119.41) .. (264,119.41) .. controls (270.47,119.41) and (276.34,121.07) .. (280.65,123.77) -- (264,135) -- cycle ; \draw   (246.06,124.64) .. controls (250.45,121.43) and (256.86,119.41) .. (264,119.41) .. controls (270.47,119.41) and (276.34,121.07) .. (280.65,123.77) ;

\draw (171,58.79) node [anchor=north west][inner sep=0.75pt]   [align=left] {1};
\draw (199,59.79) node [anchor=north west][inner sep=0.75pt]   [align=left] {1};
\draw (169,115.79) node [anchor=north west][inner sep=0.75pt]   [align=left] {1};
\draw (196,116.79) node [anchor=north west][inner sep=0.75pt]   [align=left] {1};
\draw (171,172.79) node [anchor=north west][inner sep=0.75pt]   [align=left] {1};
\draw (201,172.79) node [anchor=north west][inner sep=0.75pt]   [align=left] {1};
\end{tikzpicture}
    \end{minipage}
\end{center}
(Note that there are in general two possible geometries of the second type, given by reversing the roles of $C_1$ and $C_2$, but they behave symmetrically for the purposes of this computation, since $j=6=b-j$.)  Once again, the genus of each component is determined by Riemann--Hurwitz, and the resulting boundary strata are images of gluing maps
\[\sH_{(2,1),(g_1,0)}((1,1),(1)) \times \sH_{(2,1),(g_2,0)}((1,1),(1)) \rightarrow \sHbar_{3,g},\]
\[\sH_{3,g_1}((1,1,1)) \times \sH_{(2,1),(g_2,0)}((1,1),(1)) \rightarrow \sHbar_{3,g},\]
\[\sH_{(2,1),(g_1,0)}((1,1),(1)) \times \sH_{3,g_2}((1,1,1)) \rightarrow \sHbar_{3,g},\]
and
\[\sH_{3,g_1}((1,1,1)) \times \sH_{3,g_2}((1,1,1)) \rightarrow \sHbar_{3,g}.\]
The first of these is injective for $g_1 \neq g_2$ and a $\ZZ_2$-quotient onto its image when $g_1 = g_2$.  The second is a $\ZZ_2$-quotient for all $g_1$ and $g_2$: denoting the marked left-hand component by $(C; p_1, p_2, p_3)$ and the marked positive-genus right-hand component by $(C'; p_1', p_2')$, the pairs
\[\Big( (C; p_1, p_2, p_3), \; (C'; p_1', p_2') \Big) \;\; \text{ and } \;\; \Big( (C; p_2, p_1, p_3), \; (C'; p_2', p_1')\Big)\]
map to the same element of $\sHbar_{3,g}$.  (These two pairs are the same if there happens to exist both an automorphism on the left swapping $p_1$ and $p_2$ and an automorphism on the right swapping $p_1'$ and $p_2'$, but in this case, as above, the image in $\sHbar_{3,g}$ has an extra $\ZZ_2$ automorphism group not present in the domain of the gluing map and so the map is still a $\ZZ_2$-quotient in the orbifold sense.) The same is true of the third gluing map. Finally, the fourth gluing map above is an $S_3$-quotient in general (coming from the choice of labeling of the marked fiber on the left and corresponding labeling on the right), and an $(S_3 \times \ZZ_2)$-quotient when $g_1 = g_2$.
\end{proof}

In light of Remark~\ref{rem:disconnected_to_connected} and the fact that finite group quotients induce surjections on Chow groups, the above lemma shows that---assuming the Chow rings of these spaces satisfy the Chow--K\"unneth generation property---triviality of the Chow rings of all codimension-$1$ boundary strata in all $\sHbar_{3,g}$ will follow from triviality of the Chow rings of $\sH_{3,g}(\mu)$ for all $g \geq 0$ and all partitions $\mu$ of $k$, as well as triviality of the Chow rings of $\sH_{2,g}((2))$ for all $g \geq 1$ and $\sH_{2,g}((1,1))$ for all $g \geq 0$.  We now turn to a discussion of the techniques we will use to calculate these Chow rings.

\section{Tools for calculation}\label{tools}

\subsection{Background on Chow groups}

We review in this section some basic definitions and results we need about Chow groups. For more details, see \cite{Ful:98} and \cite{EiHa:16}.  Following these references, we will explain the background for schemes, but all of the results we mention extend to algebraic stacks by \cite{Vis:89}.

Let $X$ be a separated scheme of finite type over a fixed ground field $k$. For every nonnegative integer $p$, the group of $p$-cycles $Z_p(X)$ is defined to be the free abelian group generated by $p
$-dimensional subvarieties of $X$. The $p$th Chow group $A_p(X)$ of $X$ is defined to be the quotient of $Z_p(X)$ by the subgroup of $p$-cycles rationally equivalent to zero.
We denote $A_*(X) = \bigoplus_{p \geq 0}A_p(X)$.

The Chow groups enjoy some functorial properties that are extensively used throughout this paper.
First, there are proper pushforward and flat pullback homomorphisms on Chow groups, which are related by the compatibility property \cite[Proposition 1.7]{Ful:98}: given $f:X\to Y$ proper and $g: Y'\to Y$ flat, let $g':X'\to X$ and $f':X'\to Y'$ be the maps from the fiber product. Then $f'_*g'^*\alpha=g^*f_*\alpha$ for $\alpha\in A_*(X)$.
A useful application of the construction of flat pullback is the excision sequence \cite[Proposition 1.8]{Ful:98}: given a closed subscheme $Z \xhookrightarrow{i} X$ with open complement $U \coloneq X \smallsetminus Z \xhookrightarrow{j} X$, there exists a right exact sequence
\[A_p(Z) \xrightarrow{i_*} A_p(X) \xrightarrow{j^*} A_p(U) \rightarrow 0\]
for all $p$.  Recall also that given a rank-$r$ vector bundle $E$ on a scheme $X$ with projection $\pi : E \to X$, the flat pullback $\pi^* \colon A_{*-r}(X) \rightarrow A_*(E)$ is an isomorphism \cite[Theorem 3.3(a)]{Ful:98}.

Another key construction we will use is Chern classes: if $E$ is a locally free sheaf of rank $r$ on $X$, the $i$th Chern class is a homomorphism $c_i(E) : A_p(X) \rightarrow A_{p - i}(X)$.
In particular, when $X$ is smooth of dimension $n$, the Chow group $A_*(X)$ admits a graded ring structure with unit $[X] \in A_n(X)$, which allows us
to view the Chern class $c_i(E)$ as a cycle class in $A_{n - i}(X) = A^i(X)$,
and the Chern class homomorphism is defined as multiplication by $c_i(E)$ as an element in $A^i(X)$, the image of $[X]$ under the Chern class homomorphism.

The Chern classes $c_i(E)$ of a vector bundle $E$ over $X$ can be used to describe the Chow groups of the corresponding projective bundle $\P E$.
Let $\P E$ be the associated projective bundle with projection $\gamma \colon \P E \to X$, and denote $\zeta = c_1(\O_{\P E}(1))$.
If $X$ is smooth, then \cite[Theorem 3.3(b)]{Ful:98} can be restated as follows.

\begin{lemma}[Projective bundle formula]\label{ProjectiveChow}
    Let $X$ be a smooth scheme, and let $E$ be a rank-$r$ vector bundle on $X$.  Then the pullback $\gamma^*: A^*(X) \to A^*(\P E)$ is an injection, and there is an isomorphism
    \[
    A^*(\P E) \cong A^*(X)[\zeta]/(\zeta^{r} + c_1(E) \zeta^{r-1}  + \dots + c_{r}(E)).
    \]
    In particular, for any class $\beta \in A^*(X)$, we have $\gamma_*(\zeta^{i}\gamma^*\beta) = \beta$ when $i = r - 1$ and $0$ when $i < r - 1$.
\end{lemma}

We will often compute the fundamental class of the vanishing locus of a map of vector bundles (defined precisely below), following the discussion in Section 2.2 from \cite{CaLa:22}.
Consider the following diagram where $\tau: Y \to X$ and $\rho: V \to X$ are vector bundles and $\phi: V \to Y$ is a map of vector bundles:

\begin{center}
  \begin{tikzcd}
\rho^* Y \arrow[rd]          & Y \arrow[rd, "\tau"]                                                                                             &   \\
\rho^* V \arrow[u] \arrow[r] & V \arrow[r, "\rho"] \arrow[u, "\phi"'] \arrow[lu, "\sigma"', dashed, bend right] \arrow[l, bend left=49] & X.
\end{tikzcd}
\end{center}
  The map $\phi$ defines a section $\sigma$ of $\rho^* Y \to V$ by composing the induced map $\rho^* V \to \rho^* Y$ with the tautological section of $\rho^* V \to V$. The bundles $\rho^*V$ and $\rho^*Y$ are the pullbacks of the vector bundles $v$ and $Y$ to the total space of $V$.
  Precisely, if $(x,v)$ is a point of $V$ (i.e. $\rho(v) = x$), then $\sigma$ is the map $V \to \rho^* V \to \rho^* Y$ given by $(x,v) \mapsto (x,v,v) \mapsto (x,v, \phi(v))$. This is well-defined since $\tau (\phi (v)) = \rho(v) = x$.
By the ``vanishing locus of $\phi$" we shall mean the preimage of the zero section of $\tau$ under $\phi$.
This coincides with the vanishing locus of $\sigma$ inside the total space of $V$.

\begin{lemma}\label{topchern}
  If $\tau$ is a vector bundle of rank $r$ and $\sigma$ vanishes in codimension $r$ inside $V$, then the vanishing locus of $\sigma$ has fundamental class equal to $c_r (\rho^* Y) = \rho^* c_r (Y)$.
  Since $V \to X$ is a vector bundle, the pullback $\rho^*$ is an isomorphism of Chow rings between $A^*(X)$ and $A^*(V)$.
  Therefore, to compute the class of the vanishing locus of $\phi$, it suffices to compute the class of $c_r(Y) \in A^*(X)$.
\end{lemma}

A less standard homomorphism on Chow groups is the refined Gysin homomorphism, which we now briefly recall. Given a Cartesian diagram
\begin{equation}
\begin{tikzcd}
    g^{-1}(X) = W \arrow{d}{j}\arrow{r}{i} & V \arrow{d}{g} \\
    X \arrow{r}{f} & Y,
\end{tikzcd}
\end{equation}
where  $f$ is a regular embedding of codimension $d$ with normal bundle denoted by $N_{X/Y}$, the refined Gysin homomorphism is a map $f^! \colon A_{*}(V) \rightarrow A_{*-d}(W)$. When applied to the fundamental class of a subvariety $V' \subseteq V$, it is common to denote $f^![V']$ by $V' \cdot_Y X$, which is also called the ``refined intersection product" and is a class in $A_*(W)$.
In particular, if $V$ is of pure dimension $k$, and $i$ is also a regular embedding, then the intersection product of $V$ by $X$ over $Y$ on $W$ can be expressed as
\[V \cdot_Y X = f^![V] = \{c(j^*N_{X/Y}/N_{W/V}))\}_{k - d} \in A_{k - d}(W),\]
where we call $j^*N_{X/Y}/N_{W/V}$ the excess normal bundle of the fiber square.
The refined Gysin homomorphism has many applications.
Since the diagonal morphism $X \rightarrow X \times X$ is a regular embedding when $X$ is smooth, its refined Gysin homomorphism can be used to induce a ring structure on the Chow group $A_*(X)$.
The refined Gysin homomorphism also provides a generalization of the projection formula. Given $f : X \to Y$ a proper and flat morphism between two smooth schemes, the flat pullback gives a ring homomorphism $A^*(Y) \to A^*(X)$, and we have the projection formula $f_*(\alpha \cdot f^* \beta) = f_*\alpha \cdot \beta$ for any $\alpha \in A^*(X)$ and $\beta \in A^*(Y)$.
In fact, the map $f$ is not required to be flat; since it is always a local complete intersection morphism, we can replace the flat pullback $f^*$ by $f^!$ \cite[Proposition 8.3(c)]{Ful:98}.

Another application of the refined Gysin homomorphism is the excess intersection formula \cite[Theorem 6.3]{Ful:98}.
We state a special case of this formula that is used in Section~\ref{sec:3}
below.  Consider the Cartesian diagram
\begin{equation}\label{diag:specialExcess}
    \begin{tikzcd}
    T \ar{r}{\tau} \ar{d}{q} & S \ar{d}[swap]{p} \ar[bend left = 30]{dd}{f_2\mid_S}\\
    W \ar{r}{i} \ar{d}{j} & X_2 \ar{d}[swap]{f_2} \\
    X_1 \ar{r}{f_1} & Y,
    \end{tikzcd}
\end{equation}
where $f_1$, $f_2$, $i$ and $j$ are all regular embeddings, and $p$, $q$ are open immersions.
Assume that $X_1$ and $X_2$ are of pure dimension $k_1$ and $k_2$, and let $d_1$ and $d_2$ be the codimension of $f_1$ and $f_2$ respectively.
By \cite[Example 6.3.2]{Ful:98}, the excess normal bundle is independent of the orientation of the lower square of the diagram~\eqref{diag:specialExcess}, and thus the intersection product between $X_1$ and $X_2$ over $Y$ on $W$ is computed by
\begin{align*}
    X_1 \cdot_Y X_2 = f_1^![X_2] = \{c(j^*N_{X_1/Y}/N_{W/X_2})\}_{k_2 - d_1} = \{c(i^*N_{X_2/Y}/N_{W/X_1})\}_{k_1 - d_2} = f_2^![X_1].
\end{align*}
Applying the compatibility properties \cite[Theorem 6.2]{Ful:98} to the diagram~\eqref{diag:specialExcess}, we have
\[(f_2\circ p)^*(f_1)_*[X_1] = p^* i_* f_2^![X_1] = p^* i_* f_1^![X_2] = \tau_*(q^*f_1^![X_2]) = \tau_*(f_1^!p^*[X_2]) = \tau_*f_1^![S].\]

If we only assume $f_1$ and $f_2$ are regular embeddings with $p$ and $q$ being open immersions,
in general the fiber products $W$ and $T$ are not of pure dimension.
Let $C$ be a connected component of $T$, and restrict to the outer square of diagram~\eqref{diag:specialExcess}.
By the semicontinuity of fiber dimension, the component $C$ can have codimension at most $d_1$ in $S$.
In particular, if the restriction $\tau_C: C \rightarrow S$ of $\tau$ to the connected component $C$ is also a regular embedding, then the normal bundle $N_{C/S}$ can have rank at most $d_1$.
Let $j_C : C \rightarrow X_1$ be the restriction of $j \circ q$ to $C$.
The component $C$ is called an ``excess component'' if the excess normal bundle $j_C^*N_{X_1/Y}/N_{C/S}$ has positive rank; otherwise, $C$ is a component of expected codimension $d_1$ in $S$.
The above discussion leads to the following proposition.

\begin{proposition}\label{prop:excess}
    Assume that every connected component $C_\lambda$ of $T$
    is regularly embedded in $S$ with codimension $l_\lambda$ in $S$.
    If the regular embedding $f_1$ has codimension $d_1$, then
    \[(f_2\mid_S)^*(f_1)_*[X_1] = \sum_{\lambda} (\tau_{C_\lambda})_* \alpha_{C_\lambda} \in A^{d_1}(S),\]
    where
    \[\alpha_{C_\lambda} = \left\{c(j_{C_\lambda}^*N_{X_1/Y})c(N_{C_\lambda/S})^{-1}\right\}^{d_1 - l_\lambda} \in A^{d_1 - l_\lambda}(C_\lambda)\]
    is the part of the intersection $T$ supported on $C_\lambda$, and $j_{C_\lambda} \colon C_\lambda \rightarrow X_1$ is the restriction of $j \circ q$ to $C_\lambda$. The summation is over all connected components $C_\lambda$ of $T$.
\end{proposition}

\begin{remark}
    The above Proposition~\ref{prop:excess} is a special case of \cite[Theorem 13.9]{EiHa:16} without using the fact that $f_2|_S$ is a local complete intersection morphism.
    Since the proof of the more general statement is not included, we sketch a proof here assuming that $i'$ is a regular embedding.
    Note that the rightmost vertical map $\pi \colon X' \rightarrow X$ in \cite[Theorem 13.9]{EiHa:16} is a morphism between smooth varieties, so it is a local complete intersection morphism.
    In other words, $\pi$ admits a factorization into a regular embedding $\delta \colon X' \to Y$ followed by a smooth morphism $\rho : Y \to X$:
    \[\begin{tikzcd}
    Z' \ar{r}{i'}[swap]{\text{regular}} \ar{d}{\delta'}\ar[bend right=30]{dd}[swap]{\pi'} & X' \ar{d}[swap]{\delta} \ar[bend left = 30]{dd}{\pi} \\
    Y' \ar{r} \ar{d}{\rho'} & Y \ar{d}[swap]{\rho} \\
    Z \ar{r}{i}[swap]{\text{regular}} & X.
    \end{tikzcd}\]
    We can then use the extension of refined Gysin homomorphism to local complete intersection morphisms \cite[Section 6.6]{Ful:98}, i.e. $\pi^! = \delta^! \circ \rho'^*$.
    Since the excess normal bundle of the local complete intersection morphism $\pi$ satisfies $i'^*N_{X'/Y}/N_{Z'/Y'} = \pi'^*N_{Z/X}/N_{Z'/X'}$, it is independent of the orientation of the outer square and thus
    \begin{align*}
        \pi^*(i_*\beta) = i'_*(\pi^! (\beta)) &= i_*'\left\{\pi'^*(\beta) c(i'^*N_{X'/Y}/N_{Z'/Y'})\right\}_{b - (\dim X - \dim X')} \\
        &= i_*\left\{\pi'^*(\beta c(N_{Z/X})) c(N_{Z'/X'})^{-1}\right\}_{b - (\dim X - \dim X')}
    \end{align*}
    for any $\beta \in A_b(Z)$.
\end{remark}

\subsection{The Chow-K\"unneth generation Property}\label{subsec:ckgp}
A final aspect of Chow rings that we will need is an understanding of how they behave under product.  Specifically, we say that an algebraic stack $X$ satisfies the \emph{Chow--K\"unneth generation Property (CKgP)} if for all algebraic stacks $Y$, the tensor product map
\[A^*(X) \otimes A^*(Y) \to A^*(X \times Y)\]
is surjective. The CKgP has played an important role in recent works on intersection theory of moduli spaces \cite{CaLa:23,BaSc:23}.
One of the key features of the CKgP is that it is transferred under a wide class of morphisms. More precisely,
suppose $X \to Y$ is one of the following types of morphisms
\begin{enumerate}
\item[(M1)] an open embedding
\item[(M2)] a projective bundle
\item[(M3)] the total space of a vector bundle
\item[(M4)] a proper, surjective morphism
\item[(M5)] a gerbe banded by a finite group
\item[(M6)] a coarse moduli space morphism.
\end{enumerate}
If $Y$ has the CKgP, then $X$ has the CKgP.
In the cases (M3), (M5) and (M6), if $X$ has the CKgP, then so does $Y$.
For proofs, see \cite[Section 3.1]{CaLa:23}.

\subsection{Principal parts bundles and evaluation maps}\label{subsec:principalparts}
We recall here the definition and basic properties of relative bundles of principal parts. Suppose $\beta: X \to Y$ is a smooth, proper morphism and $\L$ is a line bundle on $X$. Let $\pi_1, \pi_2: X \times_Y X \to X$ be the two projections and let $\I$ be the ideal of the relative diagonal.
The $\emph{$m$th-order relative principal parts bundle}$ is defined as
\[\mathscr{P}^m_{X/Y}(\L) := \pi_{2*}\left(\pi_1^*\L \otimes \O_{X \times X}/\I^{m+1}\right).\]
Note that for $m = 0$, we have
$\mathscr{P}^0_{X/Y}(\L) = \L$.
The fiber of $\mathscr{P}^m_{X/Y}(\L)$ at a point $p \in X$ is naturally identified with the space of global sections of $\L$ restricted to the $m$th-order neighborhood of $p$ in its fiber $\beta^{-1}(\beta(p)) =: F$. Restricting from an $(m+1)$st-order neighborhood to an $m$th-order neighborhood induces a short exact sequence
\begin{equation} \label{filt} 0 \rightarrow \Sym^{m+1} \Omega_{X/Y} \otimes \L \rightarrow \mathscr{P}^{m+1}_{X/Y}(\L) \rightarrow \mathscr{P}^{m}_{X/Y}(\L) \rightarrow 0.
\end{equation}
The
filtration induced by repeatedly applying \eqref{filt}
corresponds to the order of vanishing of germs of sections.

There is a natural map
\begin{equation} \label{bev} \beta^*\beta_*\L \to \mathscr{P}^m_{X/Y}(\L),
\end{equation}
which we call the \emph{evaluation map}; this is because, when cohomology and base change holds for $\beta$ (e.g. if $R^1\beta_*\L = 0$), the fiber of $\beta^*\beta_*\L$ at a point $p \in X$ is naturally identified with $H^0(\L|_{F})$ and the map \eqref{bev} is given fiberwise by
\begin{equation} \label{evf} H^0(\L|_{F}) \rightarrow H^0(\L|_{F} \otimes \O_{F,p}/\mathfrak{m}_{F,p}^{m+1}).
\end{equation}
Explicitly, if $f \in H^0(\L|_F)$ is a global section, then \eqref{evf} is given by sending $f$ to its restriction to an $m$th-order neighborhood.
In the case that the fibers of
$F$ are one-dimensional, we typically write this as
\[f \mapsto (f(p), f_y(p), \tfrac{1}{2} f_{y^2}(p), \ldots, \tfrac{1}{m!} f_{y^m}(p)).\]
Here, $y$ is a local coordinate on $F$ vanishing at $p$, which we picture as a coordinate through $p$ ``in the vertical direction."

We will also encounter a case where $\beta$ factors $X \xrightarrow{a} A \to Y$ and each of these maps has one-dimensional fibers. Then there is an exact sequence
\[0 \rightarrow a^*\Omega_{A/Y} \rightarrow \Omega_{X/Y} \rightarrow \Omega_{X/A} \rightarrow 0.\]
This induces the following sequence of principal parts bundles
\begin{equation} \label{f3} 0 \rightarrow a^*\Omega_{A/Y} \otimes \L \rightarrow
\mathscr{P}^1_{X/Y}(\L) \rightarrow \mathscr{P}^1_{X/A}(\L) \rightarrow 0.
\end{equation}
In this particular case, we will write the evaluation map $\beta^*\beta_*\L \to \mathscr{P}_{X/Y}^1(\L)$ as $f \mapsto (f(p), f_y(p), f_x(p))$. The $(f(p), f_y(p))$ part corresponds to the image in the right-hand map above---the value and vertical derivative. The $f_x(p)$ part corresponds to the horizontal derivative.
The filtration \eqref{f3} shows that the
kernel of $\beta^*\beta_*\L \to \mathscr{P}^1_{X/A}(\L)$ admits a well-defined map to $a^*\Omega_{A/Y} \otimes \L$. If $f(p) = f_y(p) = 0$, we write this map as $f \mapsto f_x(p)$.

\section{Previous Results on Chow Rings of Hurwitz spaces} \label{previous}

In this section, we review previously-known results on Chow rings of Hurwitz spaces in degrees $2$ and $3$.  From here forward, we will only consider Hurwitz spaces with at most one marked ramification profile, so we denote them as, for example, $\H_{3,g}(2,1)$ instead of the more cumbersome notation $\H_{3,g}((2,1))$ used above.

\subsection{Degree 2}

Smooth curves that admit a degree-$2$ map to $\P^1$ are referred to as {\it hyperelliptic}, and their ramification points are referred to as {\it Weierstrauss points}.  These curves have been well-studied in the literature; in particular, two key facts are that a hyperelliptic curve is determined by its branch points in $\P^1$, and any hyperelliptic curve $C$ with degree-$2$ map $\alpha: C \rightarrow \P^1$ admits a {\it hyperelliptic involution} $i: C \rightarrow C$ defined by the property that $\alpha \circ i = \alpha$.  These facts can be used to quickly deduce the Chow rings of all marked Hurwitz spaces in degree $2$, as follows.

\begin{lemma} \label{h2}
The Chow rings of $\sH_{2,g}(2)$ for $g\geq 1$ and $\sH_{2,g}(1,1)$ for $g\geq 0$ are trivial and they have the CKgP.
\end{lemma}
\begin{proof}
The moduli space $\sH_{2,g}(2)$ can be identified with the moduli space of hyperelliptic curves with a marked Weierstrass point. Such curves are determined by $2g + 2$ branch points in $\pp^1$, one of which is distinguished. Consequently, the coarse moduli space of $\sH_{2,g}(2)$ is $\M_{0,2g+2}/S_{2g+1}$.
Here, $\M_{0,2g+2}$ is an open subset of $\mathbb{A}^{2g-1}$, so it has trivial Chow ring, and therefore so does the quotient.  Since we are working with rational coefficients, the Chow ring of $\sH_{2,g}(2)$ agrees with the Chow ring of its coarse moduli space.

Using the properties in Section \ref{subsec:ckgp}, this sequence of ideas also shows that $\sH_{2,g}(2)$ has the CKgP: we know that affine space has the CKgP, so $\M_{0,2g+2}$ has the CKgP, so $\M_{0,2g+2}/S_{2g+1}$ has the CKgP. Since its coarse moduli space has the CKgP, so does $\sH_{2,g}(2)$.

For $\sH_{2,g}(1,1)$, the moduli space of degree-2 covers with a marked unramified fiber, note that forgetting the second marking in the distinguished fiber defines a map from $\sH_{2,g}(1,1)$ to the moduli space of hyperelliptic curves with a marked point that is not a Weierstrass point. Because of the hyperelliptic involution $i: C \to C$, in the coarse moduli space, the pointed hyperelliptic curve $(C, p)$ is identified with $(C, \; i(p))$. Hence, a hyperelliptic curve with a non-Weierstrass point is specified by $2g + 2$ points on $\pp^1$ and one additional distinguished point. Consequently, the coarse moduli space of $\sH_{2,g}(1,1)$ is $\M_{0,2g+3}/S_{2g+2}$. A similar argument as above then shows that the Chow ring of $\sH_{2,g}(1,1)$ is trivial and $\sH_{2,g}(1,1)$ has the CKgP.
\end{proof}

\begin{remark}
With integral coefficients, the Chow rings of $\sH_{2,g}, \sH_{2,g}(2)$ and $\sH_{2,g}(1,1)$ are non-trivial. These computations are more subtle and have been carried out in \cite{EF} for $\sH_{2,g}$ for even genus $g$, \cite{FV, DL} for odd genus $g$, \cite[Theorem 1.3]{EH} for $\sH_{2,g}(2)$ and \cite[Proposition 3.11]{Landi} for $\sH_{2,g}(1,1)$.
\end{remark}

\subsection{Degree 3}\label{hprime}

Without marked ramification, the Chow rings of the Hurwitz spaces $\sH_{3,g}$ and $\sH'_{3,g}$ were calculated by Canning and the third named author.  The results are the following:

\begin{theorem}[\cite{CaLa:22}]
\label{thm:CLdeg3}
Denote by $T \in A^1(\sH'_{3,g})$ the fundamental class of the locus of covers with a triple ramification point.  Then
\[A^*(\sH'_{3,g})  = \begin{cases} \Q & \text{ if } g=2\\ \Q[T]/(T^2) & \text{ if } g \in \{3,4,5\}\\ \Q[T]/(T^3) & \text{ if } g \geq 6.
\end{cases}\]
Thus, restricting to curves with only simple ramification, we have
\[A^*(\sH_{3,g}) = \Q.\]
\end{theorem}

Before turning to Hurwitz spaces with marked ramification, we review the structure of the proof of Theorem \ref{thm:CLdeg3} in more detail. Much of the notation introduced will be relevant for our calculations in Sections \ref{sec:111}--\ref{sec:3}.

As explained in the introduction, the first step in proving Theorem~\ref{thm:CLdeg3} is to reinterpret the data of a degree-3 cover $\alpha: C \rightarrow \P^1$ in $\sH'_{3,g}$ as a pair $(E, f)$, where $E$ is a rank-2, degree-$(g+2)$ vector bundle on $\P^1$ for which the bundle $\det(E^{\vee}) \otimes \Sym^3(E)$ is globally generated, and $f$ is a section of $\det(E^{\vee}) \otimes \Sym^3(E)$.  Passing to the projectivization of $E^{\vee}$ and denoting by $\gamma: \P E^{\vee} \rightarrow \P^1$ the bundle projection, one can view
\begin{equation}
\label{eq:f}
f \in H^0(\P^1, \det(E^{\vee}) \otimes \Sym^3(E)) = H^0(\P E^{\vee}, \gamma^*\det(E^{\vee}) \otimes \O_{\P E^{\vee}}(3)),
\end{equation}
and from this perspective, the original degree-3 cover $\alpha: C \rightarrow \P^1$ is recovered by setting $C = V(f) \subseteq \P E^{\vee}$ and $\alpha = \gamma|_C$.  The result of this is a diagram
\begin{equation}
\label{eq:H'3gdiagram}
\begin{tikzcd}
\sH'_{3,g} \arrow{r}[swap]{\text{open}} & \sU \arrow{d}{\substack{\text{vector}\\ \text{bundle}}} & \\
& \sB \arrow{r}[swap]{\text{open}} & \sB_{2,g+2},
\end{tikzcd}
\end{equation}
in which $\sB_{2,g+2}$ is the moduli stack of rank-2, degree-$(g+2)$ vector bundles on $\P^1$-bundles, $\sB$ is the open substack parameterizing vector bundles $E$ for which $\det(E^{\vee}) \otimes \Sym^3(E)$ is globally generated, and, if $\pi:\sP\to \sB$ denotes the projection from the universal $\P^1$-bundle on $\sB$ and $\sE$ denotes the universal vector bundle on $\sP$,  then $\sU = \pi_*(\det(\sE^{\vee}) \otimes \Sym^3(\sE))$; note that the global generation condition is necessary in order to ensure that $\sU$ is a vector bundle.  One can view a geometric point of $\sU$ as the data of a pair $(E,f)$ as above, and inside this space, $\sH'_{3,g}$ consists of those pairs for which $V(f) \subseteq \P E^{\vee}$ is a smooth curve.

Phrased in families, an $S$-point of $\sB_{2,g+2}$ is a $\P^1$-bundle $P$ over $S$ together with a rank-2, degree-$(g+2)$ vector bundle $E$ on $P$.  The data simply of a $\P^1$-bundle over $S$ is an $S$-point of the stack $B\PGL_2$, so there is a forgetful map
\[\sB_{2,g+2} \rightarrow B\PGL_2.\]
Via this forgetful map, one can replace $\sB_{2,g+2}$ with a slight variant that we denote $\B_{2,g+2}$, defined as the fiber product of the following diagram:
\[\begin{tikzcd}
    \B_{2,g+2}\arrow[r]\arrow[d] & \sB_{2,g+2}\arrow[d]\\
B\SL_2 \arrow[r] & B\PGL_2,
\end{tikzcd}\]
in which the bottom horizontal map is induced by the natural group homomorphism $\SL_2 \rightarrow \PGL_2$.  The data of an $S$-point of $B\SL_2$ is a $\P^1$-bundle over $S$ that arises as the projectivization of a rank-2 vector bundle with trivial determinant; this is a useful perspective when working with the Chow ring of $\B_{2,g+2}$. On the other hand, because $B\SL_2 \rightarrow B\PGL_2$ is a $\mu_2$-gerbe, it follows that $\B_{2,g+2} \rightarrow \sB_{2,g+2}$ is also a $\mu_2$-gerbe, and hence
\[A^*(\B_{2,g+2}) = A^*(\sB_{2,g+2})\]
when working with rational coefficients.  For this reason, the replacement of $\sB_{2,g+2}$ by $\B_{2,g+2}$ is immaterial for our computations in the Chow ring.

Because each of the spaces in the diagram \eqref{eq:H'3gdiagram} admits a map to $\sB_{2,g+2}$ and hence to $B\PGL_2$, each of them similarly has a variant given as the fiber product over $B\SL_2 \rightarrow B\PGL_2$, which we denote by the corresponding calligraphic letters.  The resulting spaces $\H'_{3,g}$, $\U$, and $\B$ are $\mu_2$-gerbes over the original spaces and thus have isomorphic Chow rings (with rational coefficients), so we can freely replace the diagram \eqref{eq:H'3gdiagram} with
\begin{equation*}
\begin{tikzcd}
\H'_{3,g} \arrow{r}[swap]{\text{open}} & \U \arrow{d}{\substack{\text{vector}\\ \text{bundle}}} & \\
& \B \arrow{r}[swap]{\text{open}} & \B_{2,g+2},
\end{tikzcd}
\end{equation*}
for Chow ring computations.  This diagram implies, by excision and the fact that vector bundles induce isomorphisms on Chow rings, that $A^*(\H'_{3,g})$ is generated by classes pulled back from $\B_{2,g+2}$.  The ring $A^*(\B_{2,g+2})$ was calculated in \cite{Lar:23}; for future reference, we record the notation for its generators here.

To set the notation, let $\cP$ denote the universal $\P^1$-bundle on $\B$---which, by construction, is the projectivization of a rank-2 bundle $\V$ with trivial determinant---and denote by $\E$ the universal rank-2, degree-$(g+2)$ bundle on $\cP$.  We thus have a diagram
\begin{equation}
    \label{eq:PB}
    \begin{tikzcd}
        \P \E^{\vee} \arrow[r,"\gamma"] & \cP = \P \V\arrow[r,"\pi"] & \B.
    \end{tikzcd}
\end{equation}
By the projective bundle formula and the fact that $c_1(\V) = c_1(\det\V) = 0$, the pullback homomorphism $\pi^*: A^*(\B) \rightarrow A^*(\cP)$ is injective and
\[A^*(\cP) \cong \frac{A^*(\B)[z]}{(z^2 + \pi^*c_2)},\]
where
\begin{align}
\label{eq:zc2}
z &= c_1(\O_\cP(1)) \in A^1(\cP),\\
\nonumber c_2 &= c_2(\V) \in A^2(\B).
\end{align}
This implies that one can express
\begin{align}
\label{eq:aiai'}
c_1(\E) &= \pi^*(a_1) + \pi^*(a_1')z \in A^1(\cP),\\
\nonumber c_2(\E) &= \pi^*(a_2) + \pi^*(a_2')z \in A^2(\cP)
\end{align}
for uniquely-defined classes $a_i \in A^i(\B)$ and $a_i' \in A^{i-1}(\B)$.  Note, here, that $a_1' \in A^0(\B) \cong \qq$, and it can be computed explicitly from the projection formula: $a_1' = g+2$.  The results of \cite{Lar:23} imply that $A^*(\B)$ is generated by $c_2, a_1, a_2, a_2'$, so by the above discussion,
\[A^*(\H'_{3,g}) = \frac{\Q[c_2, a_1, a_2, a_2']}{R}\]
for an ideal $R$ of relations; computing this ideal explicitly through a careful excision calculus yields Theorem~\ref{thm:CLdeg3}.

We now turn to the analogous calculations of the Chow rings of degree-3 Hurwitz spaces with marked ramification, which will be computed in terms of these same generators. As explained above, it suffices to work with their base changes along $B\SL_2 \to B\PGL_2$, which we denote similarly by caligraphic letters.

\section{The Chow ring of \texorpdfstring{$\H_{3,g}(1,1,1)$}{H3g(1,1,1)}}\label{sec:111}

Similarly to the calculation of $A^*(\H_{3,g})$, we calculate the Chow ring of $\H_{3,g}(1,1,1)$ by viewing it as an open substack of a vector bundle.  The first step toward determining that vector bundle is to realize that a geometric point of $\H_{3,g}(1,1,1)$ is a tuple $(\alpha: C \rightarrow \P^1; \; p_1, p_2, p_3)$, though in fact, the third point $p_3$ is forced once the first two are chosen; thus, we can view the data of a point of $\H_{3,g}(1,1,1)$ as a rank-2, degree-$(g+2)$ vector bundle $E$ on $\P^1$ with a section $f$ as in \eqref{eq:f} and a pair of distinct unramified points $p,q \in \P E^{\vee}$ such that $f(p) = f(q) = 0$ (so that $p,q \in C = V(f)$) and $\gamma(p) = \gamma(q)$ (so that $p$ and $q$ lie in the same fiber of $\alpha = \gamma|_C$). To define a point of $\H_{3,g}(1,1,1)$, we need the additional requirement that $\alpha$ is simply branched and is not ramified at either $p$ or $q$.

With this in mind, and using the notation from \eqref{eq:PB}, we may view $\P \E^\vee \times_\cP \P \E^\vee$ as the moduli space of tuples $(\pp^1, E, p, q)$ where $(\P^1, E) \in \B$ and $p, q \in \pp E^\vee$ with $\gamma(p) = \gamma(q)$. Let $\Delta_{p=q} \subseteq \P \E^\vee \times_\cP \P \E^\vee$ be the closed substack where $p = q$, and define
\[X_{1,1,1} := (\pp \E^\vee \times_{\cP} \pp \E^\vee) \smallsetminus \Delta_{p=q}.\]
Denoting $\U = \pi_*(\det(\E^{\vee}) \otimes \Sym^3(\E))$ as above, there is a vector bundle $\U \times_\B X_{1,1,1}$ over $X_{1,1,1}$, which one can view as the moduli space of tuples $(\pp^1, E, p, q, f)$ with $(\pp^1, E, p, q) \in X_{1,1,1}$ and $f$ as in \eqref{eq:f}.  Within this vector bundle, we would like to restrict to the subspace $V_{1,1,1}$ where $f(p) = f(q) =0$; this is, in fact, a vector subbundle, but proving this requires a bit of work.

To do so, let $\eta_p: X_{1,1,1} \rightarrow  \pp \E^{\vee}$ and $\eta_q: X_{1,1,1} \rightarrow \pp \E^{\vee}$ denote the projections onto the first and second factor of $\pp \E^{\vee} \times_\cP \pp \E^{\vee}$, respectively; in this notation, we have
\[\U \times_\B X_{1,1,1} = \eta_p^*(\pi \circ \gamma)^*\U = \eta_q^*(\pi \circ \gamma)^*\U.\]
Let
\[\W := \gamma^*\det(\E^{\vee}) \otimes \O_{\pp \E^{\vee}}(3),\]
a line bundle on $\pp \E^{\vee}$, and similarly define $W$ on $\P E^\vee$.  Then there is an evaluation map
\[\text{ev}_{p,q}: \eta_p^*(\pi \circ \gamma)^*\U \rightarrow \eta_p^*\W \oplus \eta_q^*\W\]
given, in a fiber over $(\P^1, E, p, q) \in X_{1,1,1}$, as the map
\begin{align}
\label{eq:ev_fiberwise}
\nonumber    H^0(\pp E^\vee, W) &\to W|_p \oplus W|_q\\
    f &\mapsto (f(p), \; f(q))
\end{align}
that evaluates sections of the line bundle $W = \gamma^*\det (E^\vee) \otimes \O_{\pp E^\vee}(3)$ at $p$ and $q$.  The subspace $V_{1,1,1}$ of interest is precisely the preimage of zero under this evaluation map, so in order to prove that $V_{1,1,1}$ is a vector bundle, we must prove the following.

\begin{lemma} \label{ev111}
The evaluation map $\text{ev}_{p,q}: \eta_p^*(\pi \circ \gamma)^*\U \rightarrow \eta_p^*\W \oplus \eta_q^*\W$ is surjective.
\end{lemma}
\begin{proof}
It suffices to prove surjectivity fiberwise, and to do so, we factor the fiberwise map \eqref{eq:ev_fiberwise} as
\begin{equation} \label{factor} H^0(\pp E^\vee, W) \to H^0(\gamma^{-1}(\gamma(p)), \; \O_{\pp^1}(3)) \to W|_p \oplus W|_q.
\end{equation}
The first map restricts $f$ to a cubic polynomial on the fiber of $\gamma$ containing $p$, and the second map evaluates that cubic polynomial at $p$ and $q$. Since $p \neq q$, the second map is clearly surjective. Meanwhile, surjectivity of the first map follows from considering the following exact sequence of sheaves on $\pp E^\vee$:
\[0 \rightarrow W \otimes \gamma^*\O_{\pp^1}(-1) \rightarrow W \rightarrow W|_{\gamma^{-1}\gamma(p)} \rightarrow 0.\]
Applying $\gamma_*$, we obtain a sequence of sheaves on $\pp^1$:
\begin{equation}\label{p1seq}
    0\rightarrow (\gamma_*W)(-1) \rightarrow \gamma_*W \rightarrow (\gamma_*W)|_{\gamma(p)} \rightarrow 0.
\end{equation}
The key fact we need is that $H^1(\pp^1, (\gamma_*W)(-1)) = 0$, which will imply that the second map in \eqref{p1seq} induces a surjection on global sections, which is the first map in \eqref{factor}. To see this cohomological vanishing, recall that $\gamma_*W = \det (E^\vee) \otimes \Sym^3(E)$, which is globally generated by the definition of $\B \subseteq \B_{2,g+2}$. On $\pp^1$, every vector bundle splits as a direct sum of line bundles, and being globally generated is equivalent to all summands having degree $\geq 0$. It follows that every summand of $(\gamma_*W)(-1)$ has degree $\geq -1$, so $H^1(\pp^1, (\gamma_*W)(-1)) = 0$.
\end{proof}

In light of this lemma, we have a vector subbundle
\[V_{1,1,1}:= \ker(\text{ev}_{p,q}) \subseteq \U \times_\B X_{1,1,1},\]
which can be viewed as the moduli space of tuples $(\pp^1, E, p, q, f)$ where $(\P^1, E) \in \B$, $f$ is a section as in \eqref{eq:f}, and $p,q \in \pp E^{\vee}$ are distinct points such that $\gamma(p) = \gamma(q)$ and $f(p) = f(q) = 0$.  The discussion in the first paragraph of this section shows that one can realize $\H_{3,g}(1,1,1) \hookrightarrow V_{1,1,1}$ as the open substack where $V(f) \subseteq \pp E^{\vee} \rightarrow \pp^1$ is a smooth, simply-branched cover and $\gamma^{-1}(\gamma(p)) \cap V(f)$ consists of three distinct points.  To summarize, we have the following commutative diagram:
\begin{equation} \label{diagram111}
\begin{tikzcd}
\H_{3,g} \arrow{d}[swap]{\text{open}} & & \arrow{ll} \H_{3,g}(1,1,1) \arrow{d}{\text{open}} \\
\U \arrow{d}[swap]{\text{vector bundle}} & \U \times_{\B} X_{1,1,1} \arrow{d} \arrow{l} & V_{1,1,1}  \arrow{l}[swap]{\text{subbundle}} \arrow{dl}{\text{vector bundle}} \\
\B & \arrow{l} X_{1,1,1}.
\end{tikzcd}
\end{equation}

With this diagram, we can quickly verify the CKgP:
\begin{lemma} \label{ckgp111}
$\sH_{3,g}(1,1,1)$ has the CKgP.
\end{lemma}
\begin{proof}
Because the CKgP is preserved under morphisms of type (M5) in Subsection~\ref{subsec:ckgp}, and the map $\H_{3,g}(1,1,1) \to \sH_{3,g}(1,1,1)$ is a $\mu_2$-gerbe, it suffices to show that $\H_{3,g}(1,1,1)$ has the CKgP.
It was shown in the proof of \cite[Lemma 9.2]{CaLa:23} that $\B$ has the CKgP. Moreover, each of the maps
\[\H_{3,g}(1,1,1) \to V_{1,1,1} \to X_{1,1,1} \to (\pp \E^\vee \times_\cP \pp \E^\vee) \to \pp \E^\vee \to \cP \to \B\]
is a morphism allowed in (M1) -- (M3). Hence, $\H_{3,g}(1,1,1)$ has the CKgP.
\end{proof}

To compute the Chow ring of $\H_{3,g}(1,1,1)$, let
\[\zeta: = c_1(\O_{\pp \E^{\vee}}(1)) \in A^1(\pp \E^{\vee}),\]
and let
\[\zeta_p:= \eta_p^*\zeta, \;\; \zeta_q:= \eta_q^*\zeta \in A^1(X_{1,1,1}).\]
With $z$ and $a_1$ as in \eqref{eq:zc2} and \eqref{eq:aiai'}, we have the following relations in $A^*(X_{1,1,1})$.  (Here, we omit the pullbacks from the notation for classes on $X_{1,1,1}$ pulled back from $\B$ or $\cP$, as these pullbacks are injective by the projective bundle theorem.)

\begin{lemma} \label{zrel111}
We have $\zeta_p + \zeta_q  - (g+2)z - a_1 = 0 \in A^*(X_{1,1,1})$.
\end{lemma}
\begin{proof}
This relation comes from  computing the fundamental class of $\Delta_{p=q} \subseteq \pp \E^\vee \times_{\cP} \pp \E^\vee$, which lies in the complement of $X_{1,1,1}$ by definition.
To compute its class, we realize $\Delta_{p=q}$ as the vanishing locus of a section of a line bundle. Consider the tautological sequence
\begin{equation} \label{es} 0 \rightarrow \O_{\pp \E^\vee}(-1) \rightarrow \gamma^*\E^\vee \rightarrow \mathcal{Q} \rightarrow 0
\end{equation}
on $\pp \E^{\vee}$.  From this sequence, we can compute
\begin{equation} \label{cQ}
c_1(\mathcal{Q}) = \gamma^*c_1(\E^\vee) - c_1(\O_{\pp \E^\vee}(-1)) =
-(a_1 + (g+2)z) + \zeta,
\end{equation}
where the second equality follows from \eqref{eq:aiai'} and the subsequent discussion.  Now consider the composition
\[\eta_p^*\O_{\pp \E^\vee}(-1) \to \eta_p^*\gamma^*\E^\vee\cong \eta_q^*\gamma^*\E^\vee \rightarrow \eta_q^*\mathcal{Q}.
\]
This map vanishes precisely when the subspaces
$\eta_p^*\O_{\pp \E^\vee}(-1)$
and $\eta_q^*\O_{\pp \E^\vee}(-1)$ agree---that is, when $p = q$. In other words, $\Delta_{p=q}$ can be described as the vanishing locus of a section of the line bundle $\eta_p^*\O_{\pp \E^\vee}(1) \otimes \eta_q^*\mathcal{Q}$. Hence, its fundamental class is given by
\[[\Delta_{p = q}] = c_1\left(\eta_p^*\O_{\pp \E^\vee}(1) \otimes \eta_q^*\mathcal{Q}\right) =
c_1(\eta_p^*\O_{\pp \E^\vee}(1)) + c_1(\eta_q^*\mathcal{Q}) =
\zeta_p + \zeta_q  - (g+2)z - a_1. \]
By excision, the above class vanishes in the Chow ring of the complement of $\Delta_{p=q}$, which is $X_{1,1,1}$.
\end{proof}

Because $V_{1,1,1} \rightarrow X_{1,1,1}$ is a vector bundle, we can view $\zeta_p, \zeta_q \in A^*(V_{1,1,1})$.  Restricting to the open substack $\H_{3,g}(1,1,1) \subseteq V_{1,1,1}$, we find that these classes both vanish.

\begin{lemma} \label{zv}
We have $\zeta_p = \zeta_q = 0 \in A^*(\H_{3,g}(1,1,1))$.
\end{lemma}
\begin{proof}
These relations come from the fact that $p$ and $q$ are prohibited from being ramification points of $\alpha = \gamma|_{V(f)}$.  Since we already have $f(p) = 0$ in $V_{1,1,1}$, it is one additional condition for $p$ to be ramified; if $x$ and $y$ are local coordinates around $p \in \pp E^{\vee}$ in the base and fiber direction respectively under the bundle projection $\gamma: \pp E^{\vee} \rightarrow \pp^1$, the additional condition is $f_y(p) = 0$.  The divisor in $V_{1,1,1}$ defined by $f_y (p) = 0$ is the vanishing locus of the map $V_{1,1,1} \to \eta_p^*(\Omega_{\pp \E^\vee/\cP} \otimes \W)$ given by $(\P^1, E, p, q, f) \mapsto f_y(p)$.
It lies in the complement of $\H_{3,g}(1,1,1)$, and therefore by Lemma \ref{topchern}, it follows that
\begin{equation} \label{c1e} 0 = c_1(\eta_p^*(\Omega_{\pp \E^\vee/\cP} \otimes \W)) \in A^*(\H_{3,g}(1,1,1)).
\end{equation}

To compute the above Chern class, we first compute $c_1(\Omega_{\pp \E^\vee/\cP})$ using the tautological sequence \eqref{es}. In terms of the bundles introduced there, the relative Euler sequence says that $T_{\pp\E^\vee/\cP} = \mathcal{Q} \otimes \O_{\pp\E^\vee}(1)$, so $\Omega_{\pp \E^\vee/\cP} = \mathcal{Q}^\vee \otimes \O_{\pp \E^\vee}(-1)$.  Taking first Chern classes and using \eqref{cQ}, we find that
\begin{equation}
\label{eq:c1omega}
c_1(\Omega_{\pp \E^\vee/\cP}) = -2\zeta + a_1 + (g+2)z.
\end{equation}
Thus,
\[c_1(\eta_p^*(\Omega_{\pp \E^\vee/\cP} \otimes \W)) =  (-2\zeta_p + a_1 + (g+2)z) + (3\zeta_p - a_1 - (g+2)z) = \zeta_p.\]
Hence, by \eqref{c1e}, we have $\zeta_p = 0$. Since $q$ is also prohibited from being a ramification point on $\H_{3,g}(1,1,1)$, a similar argument shows $\zeta_q = 0$.
\end{proof}

Finally, from here, we can deduce that the Chow ring of $\H_{3,g}(1,1,1)$ is trivial.

\begin{lemma}\label{lem:111}
$A^*(\sH_{3,g}(1,1,1)) = \qq$.
\end{lemma}
\begin{proof}
Since the Chow rings of $\sH_{3,g}(1,1,1)$ and $\H_{3,g}(1,1,1)$ are isomorphic, it suffices to show that $A^*(\H_{3,g}(1,1,1)) = \qq$.
By the diagram in \eqref{diagram111}, there is a surjection of Chow rings
\begin{equation} \label{s1} A^*(X_{1,1,1}) \cong A^*(V_{1,1,1}) \to A^*(\H_{3,g}(1,1,1)).
\end{equation}
Thus, it suffices to show that each of the generators of $A^*(X_{1,1,1})$ is sent to zero under this map.
By the projective bundle theorem, $A^*(\pp \E^\vee \times_{\cP} \pp \E^\vee)$ is generated as an algebra over $A^*(\B)$ by $\zeta_p, \zeta_q,$ and $z$.
Using the relation in Lemma \ref{zrel111}, we can solve for $z$ in terms of $\zeta_p, \zeta_q$ and $a_1$ to see that $A^*(X_{1,1,1})$ is generated over $A^*(\B)$ by $\zeta_p$ and $\zeta_q$.
By Lemma \ref{zv}, each of $\zeta_p$ and $\zeta_q$ is sent to zero under $A^*(X_{1,1,1}) \to A^*(\H_{3,g}(1,1,1))$. Thus, it remains to show that the pullback map $A^*(\B) \to A^*(\H_{3,g}(1,1,1))$ sends each of the generators of $A^*(\B)$ to zero.
To see this, recall the diagram \eqref{diagram111}. Since the diagram commutes, the pullback map $A^*(\B) \to A^*(\H_{3,g}(1,1,1))$ factors through $A^*(\B) \to A^*(\H_{3,g})$.
However, $A^*(\H_{3,g})$ is trivial, so the pullback map $A^*(\B) \to A^*(\H_{3,g}(1,1,1))$ must send each of the generators of $A^*(\B)$ to zero. Since the pullback map \eqref{s1} is surjective, it follows that $A^i(\H_{3,g}(1,1,1)) = 0$ for all $i > 0$.
\end{proof}

\section{The Chow ring of \texorpdfstring{$\H_{3,g}(2,1)$}{H3g(2,1)}}\label{sec:21}

We follow a similar strategy to the previous section; this time, the data of a geometric point of $\H_{3,g}(2,1)$ is a rank-2, degree-$(g+2)$ vector bundle $E$ on $\P^1$ with a section $f$ as in \eqref{eq:f} and a point $p \in \P E^{\vee}$ such that $f(p) = f_y(p) = 0$ but $f_{y^2}(p) \neq 0$, where $y$ again denotes a local coordinate on $\P E^{\vee}$ in the fiber direction with respect to $\gamma: \P E^{\vee} \rightarrow \P^1$.  (Note, here, that the unramified point in the fiber $\gamma^{-1}(\gamma(p))$ is determined once the doubly-ramified point $p$ is chosen.)  Thus, in this case, we work over
\[X_{2,1} := \pp \E^\vee,\]
which can be viewed as the moduli space of tuples $(\pp^1, E, p)$ with $p \in \pp E^\vee$.

Similarly to Lemma~\ref{ev111}, we consider an evaluation map
\[\text{ev}_1:(\pi \circ \gamma)^* \U \to \mathscr{P}^1_{\pp \E^\vee/\cP}(\W)\]
of vector bundles on $X_{2,1}$; on the fiber over $(\P^1, E, p) \in X_{2,1}$, this is the map
\begin{align}
\label{eq:ev21_fiberwise}
\nonumber H^0(\pp E^\vee, W) &\to W|_{2p}\\
f &\mapsto (f(p), f_y(p))
\end{align}
that evaluates a section of the line bundle $W = \gamma^*\det (E^\vee)\otimes \O_{\pp E^\vee}(3)$ in a first-order neighborhood of $p$ contained in the vertical fiber.

\begin{lemma}
\label{ev21}
The principal parts evaluation map $\text{ev}_1:(\pi \circ \gamma)^* \U \to \mathscr{P}^1_{\pp \E^\vee/\cP}(\W)$ of vector bundles on $X_{2,1}$ is surjective.
\end{lemma}
\begin{proof}
Similarly to the proof of Lemma \ref{ev111}, the fiberwise map \eqref{eq:ev21_fiberwise} map factors as
\begin{equation} \label{factor2} H^0(\pp E^\vee, W) \to H^0(\gamma^{-1}\gamma(p), \O_{\pp^1}(3)) \to W|_{2p}.
\end{equation}
In fact, we have already shown that the first map is surjective in the proof of Lemma \ref{ev111}. Meanwhile, the second map is surjective because it sends a degree-three polynomial on $\pp^1$ to its restriction to a first-order neighborhood of $p$.
\end{proof}

Since the principal parts evaluation map is surjective, its kernel is a vector bundle. We define $V_{2,1}$ to be the total space of the kernel:
\[V_{2,1} := \ker(\text{ev}_1)\subseteq \U \times_\B X_{2,1}.\]
Thinking of $\U$ as the moduli space of tuples $(\pp^1, E, f)$, we can describe $V_{2,1}$ as the moduli space of tuples $(\pp^1, E, p, f)$ such that $f(p) = f_y(p) = 0$, i.e., $p$ is a ramification point of $V(f) \to \pp^1$.

By the discussion at the beginning of the section, there is a natural map $\H_{3,g}(2,1) \to V_{2,1}$, which realizes $\H_{3,g}(2,1)$ as the open substack of $V_{2,1}$ where $V(f) \subseteq \pp E^\vee \to \pp^1$ is a smooth, simply-branched cover.  In summary, then, we have a commutative diagram:
\begin{equation} \label{diagram21}
\begin{tikzcd}
\H_{3,g} \arrow{d}[swap]{\text{open}} & & \arrow{ll} \H_{3,g}(2,1) \arrow{d}{\text{open}} \\
\U \arrow{d}[swap]{\text{vector bundle}} & \U \times_{\B} X_{2,1} \arrow{d} \arrow{l} & V_{2,1}  \arrow{l}[swap]{\text{subbundle}} \arrow{dl}{\text{vector bundle}} \\
\B & \arrow{l} X_{2,1}.
\end{tikzcd}
\end{equation}

An argument very similar to Lemma \ref{ckgp111} shows the following.
\begin{lemma} \label{ckgp21}
$\sH_{3,g}(2,1)$ has the CKgP.
\end{lemma}

Similarly to the two relations in $A^*(\H_{3,g}(1,1,1))$ given by Lemma~\ref{zv}, the following two lemmas give relations in $A^*(\H_{3,g}(2,1))$.

\begin{lemma}\label{zrel21-1}
We have $-\zeta + a_1 + (g+2)z = 0 \in A^*(\H_{3,g}(2,1))$.
\end{lemma}
\begin{proof}
This relation comes from the fact that $p$ is prohibited from being a point of triple ramification. Inside $V_{2,1}$, we already have $f(p) = f_y(p) = 0$. It is therefore one additional condition for $p$ to be triply ramified, given by $f_{y^2}(p) = 0$.
The divisor in $V_{2,1}$ defined by $f_{y^2}(p) = 0$ is the vanishing locus of the map $V_{2,1} \to \Omega_{\pp \E^\vee/\cP}^{\otimes 2} \otimes \W$ sending $(\P^1, E, p, f) \mapsto f_{y^2}(p)$. This closed locus lies in the complement of $\H_{3,g}(2,1)$, so its fundamental class vanishes in $\H_{3,2}(2,1)$. Applying Lemma \ref{topchern} and equation \eqref{eq:c1omega} to compute this fundamental class, it follows that
\[0 = c_1(\Omega_{\pp \E^\vee/\cP}^{\otimes 2} \otimes \W) =
2(-2\zeta + a_1 + (g+2)z) + (3\zeta - a_1 - (g+2)z)
\in A^*(\H_{3,g}(2,1)). \qedhere\]
\end{proof}

\begin{lemma}\label{zrel21-2}
We have $3\zeta - a_1 - (g + 4)z = 0 \in A^*(\H_{3,g}(2,1))$.
\end{lemma}
\begin{proof}
This relation comes from the fact that $p$ is prohibited from being a simple node (or worse). It is one additional condition given by $f_{x}(p) = 0$ added to $f(p) = f_y(p) = 0$ in $V_{2,1}$.
The divisor in $V_{2,1}$ defined by $f_x(p) = 0$ is the vanishing locus of the map $V_{2,1} \to \gamma^*\Omega_{\cP/\B}\otimes \W$ sending $(\P^1, E, p, f) \mapsto f_x(p)$ (well-defined by \eqref{f3}). This closed locus lies in the complement of $\H_{3,g} (2,1)$, so its fundamental class vanishes in $\H_{3,2}(2,1)$. Applying Lemma \ref{topchern} to compute this fundamental class, it follows that
\[0 = c_1(\gamma^*\Omega_{\cP/\B} \otimes \W) = (-2z) + (3\zeta - a_1 - (g + 2)z) \in A^*(\H_{3,g}(2,1)). \qedhere\]
\end{proof}

Combining these, we deduce the triviality of the Chow ring in this case.

\begin{lemma}\label{lem:21}
    $A^*(\sH_{3,g}(2,1)) = \qq$.
\end{lemma}
\begin{proof}
It suffices to show that $A^*(\H_{3,g}(2,1)) = \qq$. By the diagram in \eqref{diagram21}, there is a flat pullback map on Chow rings which is surjective
\begin{equation}\label{s21}
    A^*(X_{2,1}) \cong A^*(V_{2,1}) \to A^*(\H_{3,g}(2,1)).
\end{equation}
Thus, it suffices to show that each of the generators of $A^*(X_{2,1})$ is sent to zero under this pullback map. Note that $A^*(X_{2,1}) = A^*(\pp\E^\vee)$ is generated as an algebra over $A^*(\B)$ by $\zeta$ and $z$ by the projective bundle theorem. Moreover, upon pullback to $\H_{3,g}(2,1)$, the classes $\zeta$ and $z$ can be solved for in terms of $a_1$ using the relations in Lemma \ref{zrel21-1} and \ref{zrel21-2}; specifically, we have $\zeta = z = \left(-\frac{1}{g + 1}\right)a_1$.  Similarly to Lemma~\ref{lem:111}, the fact that $A^*(\H_{3,g})$ is trivial implies that the pullback map $A^*(\B) \rightarrow A^*(\H_{3,g}(1,1,1))$ sends all elements of $A^*(\B)$ (including, in particular, $a_1$) to zero, so \eqref{s21} indeed sends all generators to zero.
\end{proof}

\section{The Chow ring of \texorpdfstring{$\H_{3,g}(2)$}{H3g(3)}}\label{sec:3}
Although the initial set-up of this section is very similar to the previous one, the excision calculations are much more intricate.

The first step is to notice that the data of a geometric point of $\H_{3,g}(3)$ is a rank-2, degree-$(g+2)$ vector bundle $E$ on $\P^1$ with a section $f$ as in \eqref{eq:f} and a point $p \in \P E^{\vee}$ such that $f(p) = f_y(p) = f_{y^2}(p) = 0$.  To capture this, we build a vector bundle over
\[X_3 := \pp \E^\vee\]
using a second-order principal parts bundle.  Similarly to Lemmas~\ref{ev111} and \ref{ev21}, the first step is the following.

\begin{lemma}\label{ev3}
The principal parts evaluation map $\text{ev}_2:(\pi \circ \gamma)^* \U \to \mathscr{P}^2_{\pp \E^\vee/\cP}(\W)$ of vector bundles on $X_{3}$ is surjective.
\end{lemma}
\begin{proof}
On the fibers over $(\pp^1, E, p) \in X_{3}$, this is the map
\begin{align*}
    H^0(\pp E^\vee, W) &\to W|_{3p}\\
    f &\mapsto (f(p), f_y(p), \tfrac{1}{2}f_{y^2}(p))
\end{align*}
that evaluates a section of the line bundle $W = \O_{\pp E^\vee}(3) \otimes \gamma^*\det E^\vee$ in a second-order neighborhood of $p$ contained in the vertical fiber.  Similarly to the proof of Lemma \ref{ev111}, this map factors as
\begin{equation} \label{factor3} H^0(\pp E^\vee, W) \to H^0(\gamma^{-1}\gamma(p), \O_{\pp^1}(3)) \to W|_{3p}.
\end{equation}
We have already shown that the first map is surjective in the proof of Lemma \ref{ev111}. Meanwhile, the second map sends a degree-$3$ polynomial on $\pp^1$ to its restriction to a second-order neighborhood of $p$. This is surjective because the value, first-, and second-order derivatives of a cubic polynomial can simultaneously attain any three values.
\end{proof}

Since the principal parts evaluation defined in Lemma \ref{ev3} is surjective, its kernel is a vector bundle. We define
\[V_{3} := \ker(\text{ev}_2)\subseteq \U \times_\B X_3.\]
We can view $V_3$ as the moduli space of tuples $(\pp^1, E, p, f)$ such that $f(p) = f_y(p) = f_{y^2}(p) = 0$, i.e., $p$ is a triple ramification point of $\gamma|_{V(f)}$.

There is a natural map $\H_{3,g}'(3) \rightarrow V_3$ realizing $\H_{3,g}'(3)$ as the open substack of $V_3$ where $\gamma|_{V(f)}$ is a smooth triple cover; such covers are triply-ramified at $p$ and are allowed arbitrary ramification elsewhere.  If we impose the open condition that all other ramification points besides $p$ are simple ramification points, this defines the open substack $\H_{3,g}(3) \subseteq \H_{3,g}'(3) \subseteq V_3$.

An argument similar to Lemma \ref{ckgp111} shows the following.

\begin{lemma}\label{ckgp3}
Both $\sH_{3,g}'(3)$ and $\sH_{3,g}(3)$ have the CKgP.
\end{lemma}

We denote
\[\zeta_p := c_1(\O_{\pp \E^{\vee}}(1)) \in A^1(X_3).\]
(This is the same as the class $\zeta$ introduced earlier, but we use the notation $\zeta_p$ here to stress the role of the point $p$, as a second similar point will appear later in the computation.)  By the projective bundle theorem, $A^*(X_3)$ is generated as an algebra over $A^*(\B)$ by $\zeta_p$ and $z$. Because $\H_{3,g}(3)$ is an open substack of a vector bundle over $X_3$ by the discussion above, these classes also generate $A^*(\H_{3,g}(3))$ as an algebra over $A^*(\B)$.  Hence, $A^*(\H_{3,g}(3))$ is generated as a $\qq$-algebra by $a_1, a_2', a_2, c_2, \zeta_p$, and $z$.

By \cite[Theorem 1.1(1)]{CaLa:22}, we know that
$A^*(\H_{3,g}')$ is generated as a $\qq$-algebra by $a_1$ and $a_2'$, so there are relations on $\H_{3,g}'$ that express $a_2$ and $c_2$ in terms of these generators. These relations pull back along the forgetful map $\H_{3,g}'(3) \to \H_{3,g}'$ (see the top arrow in the diagram \eqref{diagram3} below), so $A^*(\H_{3,g}'(3))$ is generated as a $\qq$-algebra by $\zeta_p$, $z$, $a_1$ and $a_2'$. By excision, these same classes generate $A^*(\H_{3,g}(3))$ as a $\qq$-algebra.

In order to show $A^*(\H_{3,g}(3)) = \Q$, it thus suffices to find four independent codimension-$1$ classes in terms of the aforementioned generators that lie in the kernel of the surjection $A^*(X_3) \rightarrow A^*(\H_{3,g}(3))$.  Three of these are provided by the following three lemmas.

\begin{lemma}\label{rel3}
We have $(8g + 12)a_1 - 9a_2' = 0 \in A^*(\H_{3,g}(3))$.
\end{lemma}
\begin{proof}
Let $\Delta_{3,g} \coloneqq \U \smallsetminus \H_{3,g}'$. It is computed in \cite[Lemma 4.2]{CaLa:22} that
\[[\Delta_{3,g}] = (8g + 12)a_1 - 9a_2'\in A^*(\U).\]
By excision, the pullback of this class to $A^*(\H'_{3,g})$ vanishes, and therefore its pullback to $A^*(\H_{3,g}(3))$ vanishes, as well.
\end{proof}

\begin{lemma} \label{zrel3-1}
We have $-3\zeta_p + 2a_1 + 2(g + 2)z = 0 \in A^*(\H_{3,g}(3))$.
\end{lemma}
\begin{proof}
This relation comes from the fact that $p$ is prohibited from being a point at which the curve $V(f)$ has order of contact to the vertical ruling $L_p = \gamma^{-1}(\gamma(p))$ of $\pp E^\vee$ at least $4$, since this can only be the case if $V(f)$ contains $L_p$ as a component and is therefore not smooth.  Alongside the conditions  $f(p) = f_y(p) = f_{y^2}(p) = 0$ defining $V_{3}$, it is one additional condition for $p$ to have order of contact at least 4, given by $f_{y^3}(p) = 0$.  The divisor in $V_{3}$ defined by $f_{y^3}(p) = 0$ is the vanishing locus of the map $V_{3} \to \Omega_{\pp\E^\vee/\cP}^{\otimes 3} \otimes \W$ sending $(\P^1, E, p, f) \mapsto f_{y^3}(p)$. This closed locus lies in the complement of $\H_{3,g} (3)$. Thus by Lemma \ref{topchern} and equation \eqref{eq:c1omega}, it follows that
\[0 = c_1(\Omega_{\pp \E^\vee/\cP}^{\otimes 3} \otimes \W) = 3(-2\zeta_p + a_1 + (g+2)z) + (3\zeta_p - a_1 - (g+2)z) \in A^*(\H_{3,g}(3)). \qedhere\]
\end{proof}

\begin{lemma}\label{zrel3-2}
We have $3\zeta_p - a_1 - (g + 4)z = 0 \in A^*(\H_{3,g}(3))$.
\end{lemma}
\begin{proof}
This relation comes from the fact that $p$ is prohibited from being a singular point.  Similarly to the proof of Lemma \ref{zrel21-2}, it is one additional condition on $V_3$ that $p$ be a singular point, given by $f_x(p) = 0$, where $x$ again denotes a local coordinate on $\P E^{\vee}$ in the base direction with respect to the bundle projection $\gamma$.    The divisor in $V_3$ defined by $f_x(p) = 0$ is the vanishing locus of the map $V_3 \rightarrow \gamma^*\Omega_{\cP/\B} \otimes \W$ given by sending $(\P^1, E, p, f) \mapsto f_x(p)$. Thus, again by Lemma \ref{topchern} and equation \eqref{eq:c1omega}, it follows that
\[0 = c_1(\gamma^*\Omega_{\cP/\B} \otimes \W) = -2z + (3\zeta_p - a_1 - (g+2)z) \in A^*(\H_{3,g}(3)). \qedhere\]
\end{proof}

To find an additional relation in $A^*(\H_{3,g}(3))$, we consider the ``triple-triple'' locus
\[TT \coloneqq \H_{3,g}'(3) \smallsetminus \H_{3,g}(3)\]
consisting of covers with a second triple ramification point aside from the marked one, whose class $[TT]$ vanishes in $A^*(\H_{3,g}(3))$ by excision. To describe $TT$, we work over
\[\widetilde{X}_3 \coloneqq X_3 \times_\B \pp \E^\vee = \pp \E^\vee \times_\B \pp \E^\vee,\]
which is the moduli space of tuples $(\pp^1, E, p, q)$ with $p, q \in \P E^{\vee}$. We denote by $\Delta_{p = q}$ the closed substack of $\widetilde{X}_3$ defined by the condition $p = q$, and by $\Delta_{\gamma(p) = \gamma(q)}$ the closed substack defined by the condition $\gamma(p) = \gamma(q)$.

We consider the following commutative diagram:
\begin{equation}\label{diagram3}
\begin{tikzcd}
\H_{3,g}' \arrow{d}[swap]{\text{open}} & & \arrow{ll} \H_{3,g}'(3) \arrow{d}[swap]{\text{open}} & S \arrow{l}[swap]{\tilde{\eta}_p^\circ} \arrow{d}[swap]{\text{open}} \arrow[bend left = 50]{dd}{\phi} \\
\U \arrow{d}[swap]{\text{vector bundle}} & \U \times_{\B} X_{3} \arrow{d} \arrow{l} & V_{3}  \arrow{l}[swap]{\text{subbundle}} \arrow{dl}{\text{vector bundle}} & \eta_p^* V_3 \arrow{l}{\tilde{\eta}_p} \arrow{d}{\rho''} \\
\B & \arrow{l} X_{3} &&  \arrow{ll}{\eta_p} \widetilde{X}_3.
\end{tikzcd}
\end{equation}
Above, $\eta_p: \widetilde{X}_3 \to X_3$ is projection onto the first factor. The two rightmost quadrilaterals are Cartesian, so in particular, \[S = \tilde{\eta}_p^{-1} (\H_{3,g}'(3)),\]
which we can view as the moduli space of tuples $(\P^1, E, f, p, q)$ such that $V(f)$ is a smooth triple cover with a triple ramification point at $p$.
Below, we will define a locus $Z \subseteq S$ that parameterizes triple covers for which $p$ and $q$ are distinct and are both triple ramification points. The closure
will be a codimension-$3$ cycle $\Zbar \subseteq S$ such that $(\tilde{\eta}_p^\circ)_*[\Zbar] = [TT] \in A^1(\H_{3,g}'(3))$.
Our goal is thus to compute the class of $[\Zbar]$ and then this pushforward.

Let $\eta_q: \widetilde{X}_3 \to \pp \E^\vee$ be the projection onto the second factor. There are natural inclusions of $\eta_p^*V_3$ and $\eta_q^*V_3$
into the total space
\[\eta_p^*(\U\times_B X_3) = \eta_p^*\gamma^*\pi^*\U = \eta_q^*\gamma^*\pi^*\U.\]
If we think of the above total space as parameterizing $(\pp^1, E, p, q, f)$, then  $\eta_p^*V_3$ is the subspace defined by $f(p) = f_y(p) = f_{y^2}(p) = 0$ and $\eta_q^*V_3$ is the subspace defined by $f(q) = f_y(q) = f_{y^2}(q) = 0$. We are interested in their intersection:
\begin{equation}
\begin{tikzcd}
\iota_p^{-1}(\eta_q^*V_3) \arrow{r}\arrow{d} & \eta_p^* V_3 \arrow{d}{\text{regular}}[swap]{\iota_p} \\
\eta_q^* V_3 \arrow{r}{\text{regular}}[swap]{\iota_q} & \eta_p^*(\U \times_\B X_3).
\end{tikzcd}
\end{equation}
It is worth noting that both $\iota_p$ and $\iota_q$ are regular embeddings of codimension $3$ with normal bundles denoted by $N_{\iota_p}$ and $N_{\iota_q}$ respectively.
The intersection $\iota_p^{-1}(\eta_q^*V_3)$ parameterizes $(\pp^1, E, p, q, f)$ such that
$f(p) = f_y(p) = f_{y^2}(p) = 0$ and $f(q) = f_y(q) = f_{y^2}(q) = 0$. This intersection has multiple components, and we are interested in understanding these components after restricting to $S \subseteq \eta_p^*V_3$.  To this end, we define
\[Y := \phi^{-1}(\Delta_{\gamma(p)=\gamma(q)})\cap \iota_p^{-1}(\eta_q^*V_3) \subseteq S,\]
\[Z := \phi^{-1}(\Delta_{\gamma(p) = \gamma(q)}^c) \cap \iota_p^{-1}(\eta_q^*V_3) \subseteq S.\]
As $S = \phi^{-1}(\Delta_{\gamma(p) = \gamma(q)}^c) \sqcup \phi^{-1}(\Delta_{\gamma(p)=\gamma(q)})$
by construction, we then have a diagram
\begin{equation}\label{diagram3-excess}
\begin{tikzcd}
Y \sqcup Z \arrow{r} \arrow{d} & S \arrow[bend left = 60]{dd}{\iota_p|_S} \arrow{d} \\
\iota_p^{-1}(\eta_q^*V_3) \arrow{r}\arrow{d} & \eta_p^* V_3 \arrow{d}{\iota_p}[swap]{\text{regular}} \\
\eta_q^* V_3 \arrow{r}{\text{regular}}[swap]{\iota_q} & \eta_p^*(\U \times_\B X_3),
\end{tikzcd}
\end{equation}
where both squares are Cartesian.  As we shall see below, $Y$ and $\overline{Z}$ are the components of the restriction to $S$ of the intersection $\iota_p^{-1}(\eta_q^*V_3)$. Note that $Y$ is already closed.

The expected codimension in $S$ of $\iota_p^{-1}(\eta_q^*V_3) \cap S$ is $3$, since $S$ is open in $\eta_p^*V_3$ and intersecting $\eta_p^*V_3$ with $\eta_q^*V_3$ imposes $3$ additional constraints on $\eta_p^*V_3$.  In Lemmas~\ref{lem:codimY} and \ref{lem:codimZ} below, we will show that $\overline{Z}$ has the expected codimension, whereas the codimension of $Y$ is smaller.  From here, the excess intersection formula will allow us to calculate the contribution of $Y$ to the intersection, which we can subtract off to calculate the desired class $[\overline{Z}]$.

\begin{lemma}
\label{lem:codimY}
We have $Y = \phi^{-1}(\Delta_{p = q})$, which has codimension $2$ in $S$.
\end{lemma}
\begin{proof}
By definition, $Y \subseteq S$ consists of tuples $(\pp^1, E, p, q, f)$ such that $V(f)$ is a smooth curve, $\gamma(p) = \gamma(q)$ and
\[f(p) = f_y(p) = f_{y^2}(p) = f(q) = f_y(q) = f_{y^2}(q) = 0.\]
It is clear that $\phi^{-1}(\Delta_{p=q})$ is contained in $Y$. Suppose for contradiction that
$Y$ is not contained in $\phi^{-1}(\Delta_{p=q})$. Then there exists $(\pp^1, E, p, q, f)$ satisfying the conditions above with $p \neq q$. However, this is impossible since $V(f)$ would then have contact order $6$ with the fiber containing $p$ and $q$. This would imply that $V(f)$ contains that vertical fiber, but this forces $V(f)$ to be singular.
\end{proof}

\begin{lemma}
\label{lem:codimZ}
$Z$ has codimension $3$ in $S$.
\end{lemma}
\begin{proof}
We study the projection of $Z$ to $\Delta^c_{\gamma(p) = \gamma(q)}$.
Since $S$ is open in $\eta_p^*V_3$ and $\eta_p^*V_3$ is codimension 3 in $\eta_p^*\gamma^*\pi^*\U$, the fibers of $S \to \widetilde{X}_3$ are empty or codimension $3$ in the fibers of $\eta_p^*\gamma^*\pi^*\U \to \widetilde{X}_3$.
We will show that all fibers of $Z \to \Delta^c_{\gamma(p) = \gamma(q)}$
are either empty or codimension $6$ in the fibers of \[\eta_p^*\gamma^*\pi^*\U|_{\Delta^c_{\gamma(p) = \gamma(q)} } \to \Delta^c_{\gamma(p) = \gamma(q)} \subseteq \widetilde{X}_3.\]

The fiber of $Z \to \Delta^c_{\gamma(p) = \gamma(q)}$ over a point
$(\pp^1, E, p, q)$ is, by construction, the open subset of $f$ in the kernel of the evaluation map
\begin{align}
\label{kev}
H^0(\pp E^\vee, W) &\to W|_{3p} \oplus W|_{3q}\\
\nonumber f &\mapsto (f(p), f_y(p), \tfrac{1}{2} f_{y^2}(p), f(q), f_y(q), \tfrac{1}{2} f_{y^2}(q))
\end{align}
where $V(f)$ is smooth.  Since the codomain of the above map is 6-dimensional, it suffices to show that either the map is surjective or the kernel consists entirely of equations defining singular curves.

We treat two cases depending on the splitting type of $E$. Let us write $E = \O(m) \oplus \O(n)$ where $m \leq n$ and $m + n = g+2$.  Then $\Sym^3E \otimes \det E^\vee$ splits as
\begin{align} \Sym^3 E \otimes \det E^\vee &= (\O(3m) \oplus \O(2m + n) \oplus \O(m + 2n) \oplus \O(3n)) \otimes \O(-m-n) \label{Sym}\\
&= \O(2m - n) \oplus \O(m) \oplus \O(n) \oplus \O(2n - m). \notag
\end{align}
First, suppose that $2m - n \geq 1$, so that every summand of $\Sym^3 E \otimes \det E^\vee$ has degree $1$ or more.  Then we can factor the map  \eqref{kev} as
\begin{equation} \label{compo} H^0(\pp E^\vee, W) \to H^0(\gamma^{-1}\gamma(p), \O_{\pp^1}(3)) \oplus
H^0(\gamma^{-1}\gamma(q), \O_{\pp^1}(3)) \rightarrow W|_{3p} \oplus W|_{3q},
\end{equation}
and the second map is surjective.  We will prove that the first map is also surjective, so that \eqref{kev} is surjective.  The surjectivity of the first map in \eqref{compo} comes from considering the following sequence of sheaves on $\pp E^\vee$:
\[0 \rightarrow W \otimes \gamma^*\O_{\pp^1}(-2) \rightarrow W \rightarrow W|_{\gamma^{-1}\gamma(p)} \oplus W|_{\gamma^{-1}\gamma(q)} \rightarrow 0.\]
As in the proof of Lemma \ref{ev111}, it suffices to see that $0 = H^1(\pp^1, (\gamma_*W)(-2))$. Since $\gamma_*W = \Sym^3 E \otimes \det E^\vee$ and we assume all summands of this bundle have degree $\geq 1$, its twist down by $2$ has no $H^1$.

It remains to treat the case when $2m - n \leq 0$.
In this case, it turns out that the map
\begin{equation} \label{fmap} H^0(\pp E^\vee, W) \to H^0(\gamma^{-1}\gamma(p), \O_{\pp^1}(3)) \oplus
H^0(\gamma^{-1}\gamma(q), \O_{\pp^1}(3))
\end{equation}
is not surjective, as we now explain.
Since $p$ and $q$ are not in the same fiber of $\gamma: \P E^{\vee} \rightarrow \P^1$, we can choose a coordinate $x$ on the base $\P^1$ in which $\gamma(p) = 0$ and $\gamma(q) = 1$.

Next, we choose relative homogeneous coordinates $[Y_0:Y_1]$ on $\pp E^\vee$; more precisely, we will have
\begin{align*}
&Y_0 \in H^0(\P E^{\vee}, \O_{\P E^{\vee}}(1) \otimes \gamma^*\O_{\P^1}(-n)),\\
&Y_1 \in H^0(\P E^{\vee}, \O_{\P E^{\vee}}(1) \otimes \gamma^*\O_{\P^1}(-m)).
\end{align*}
To choose these coordinates, first note that each of the two summands in $E^\vee = \O(-m) \oplus \O(-n)$ defines a line in each fiber of $E^{\vee} \rightarrow \pp^1$, or in other words, a section of $\pp E^\vee \to \pp^1$. The section corresponding to $\O(-m)$ is called the
\emph{directrix} and is distinguished in that its image is the unique curve in $\P E^{\vee}$ of self-intersection $m - n < 0$. (The other section is not distinguished and depends on our choice of splitting for $E$.)
We define the coordinate $Y_0$ so that $V(Y_0)$ is the directrix and $Y_1$ so that $V(Y_1)$ is the other section corresponding to our choice of $\O(-n)$ summand.
With these coordinates on $\pp E^\vee$, any section $f \in H^0(\pp E^\vee, W)$ can be written as
\begin{equation} \label{feq} f = \gamma^* a \cdot Y_1^3 + \gamma^*b \cdot Y_1^2Y_0 + \gamma^*c \cdot Y_1Y_0^2 + \gamma^* d \cdot Y_0^3,
\end{equation}
where $a, b, c, d$ are sections of the summands of  $\Sym^3 E \otimes \det E^\vee$ as in \eqref{Sym}, i.e. polynomials in $x$ of degrees $2m - n, m, n,$ and $2n - m$ respectively.

Note that if $2m - n < 0$, then $a(x)$ would be identically zero, which forces $V(f)$ to be reducible. Thus, it suffices to consider the case $2m  - n = 0$. In this case, $a(x)$ has degree $0$ while $b(x), c(x),d(x)$ have positive degree. In terms of \eqref{feq}, the map \eqref{fmap} is given explicitly by
\[f \mapsto (a(0)Y_1^3 + b(0)Y_1^2Y_0 + c(0)Y_1Y_0^2 + d(0)Y_0^3, a(1)Y_1^3 + b(1)Y_1^2Y_0 + c(1)Y_1Y_0^2 + d(1)Y_0^3).
\]
Since $a(x)$ is constant, it follows that the image of \eqref{fmap} is those pairs of degree-$3$ homogeneous polynomials in $Y_0$ and $Y_1$ with the same coefficient of $Y_1^3$. Despite the fact that \eqref{fmap} is not surjective, we will still prove our original goal that either \eqref{kev} is surjective or the kernel consists entirely of equations defining singular curves.

First, suppose that one of $p$ or $q$ lies on the directrix of $\pp E^\vee$. Without loss of generality, say $p$ lies on the directrix, so $p$ is the point $(x, [Y_0:Y_1]) = (0, [0:1])$.
In order for $f$ to vanish at $(0, [0:1])$, we must have $a(0) = 0$. But $a$ is a section of $\O(2m - n) = \O$, which is to say $a$ is a constant, so $a$ is identically $0$. In this case, $V(f)$ must be singular. Thus, in this case, the kernel of \eqref{kev} consists entirely of equations defining singular curves.

Next, suppose neither $p$ nor $q$ lies on the directrix of $\pp E^\vee$, so that $Y_0(p) \neq 0$ and $Y_0(q) \neq 0$.  Then $y = Y_1/Y_0$ is a local coordinate on the vertical fiber that vanishes at $p$ and takes some nonzero value $y_1$ at $q$.  In the coordinate $(x,y)$ on $\P E^{\vee}$, then, we have $p = (0,0)$ and $q = (1,y_1)$, whereas $f \in H^0(\P E^{\vee}, W)$ is given by
\[f(x,y) = a(x) y^3 + b(x) y^2 + c(x) y + d(x).\]
Thus, in these coordinates, the map \eqref{kev} is given by
\[f\mapsto (d(0), c(0), b(0), a(1)y_1^3+b(1)y_1^2+c(1)y_1+d(1), 3a(1)y_1^2+b(1)y_1+c(1), 3a(1)y_1+b(1)).\]
This is surjective, because once $d(0), c(0)$, and $b(0)$ are chosen, the values of $a(1), b(1), c(1), d(1)$ can still be chosen arbitrarily given that $b,c$, and $d$ have positive degree.  This confirms the surjectivity of \eqref{kev} in this case and therefore completes the proof.
\end{proof}

\begin{lemma}
The intersection $\Zbar \cap Y$ has codimension at least $4$ in $S$.
\end{lemma}
\begin{proof}
Since $Y$ is disjoint from $Z$, we have $\overline{Z} \cap Y \subset \overline{Z} \smallsetminus Z$. The space $\overline{Z} \smallsetminus Z$ has dimension strictly less than the dimension of $Z$.
Since $Z$ has codimension $3$, we conclude $\overline{Z} \smallsetminus Z$ has codimension $4$ or more.
\end{proof}
\begin{remark}
In fact, we believe that $\overline{Z} \cap Y$ is empty, or in other words, that $Z$ is already closed in $S$. However, all that matters for the proof is that $\overline{Z} \cap Y$ has codimension $> 3$ in $S$.
\end{remark}

Recall, now, that our goal is to calculate the class of $\overline{Z} \subseteq S$, since it will satisfy
\[(\tilde{\eta}_p^\circ)_*[\Zbar] = [TT] \in A^1(\H_{3,g}'(3)),\]
thereby allowing us to calculate the class $[TT]$. Because of the excess component $Y$, we must use the excess intersection formula.
By excision, calculations in $A^3(S)$ can be performed on the complement of a locus of codimension $> 3$. In particular, we can work on the complement of $\overline{Z} \cap Y$, so we can assume $\overline{Z}$ and $Y$ are disjoint components of $\iota_p^{-1}(\eta_q^*V_3) \cap S \subseteq S$.

Let $\iota_Y$ be the inclusion map from $Y$ to $S$.  The excess intersection formula (see Proposition \ref{prop:excess}) applied to the diagram~\eqref{diagram3-excess} reads
\[(\iota_p\mid_S)^*(\iota_q)_*[\eta_q^*V_3] = [\overline{Z}] + (\iota_{Y})_*\alpha_Y.\]
Pushing forward along the proper map $\tilde{\eta}_p^\circ \colon S \to \H_{3,g}'(3)$, we have
\begin{equation}\label{excess-formula}
    (\tilde{\eta}_p^\circ)_*(\iota_p\mid_S)^*(\iota_q)_*[\eta_q^* V_3] = (\tilde{\eta}_p^\circ)_*[\overline{Z}] + (\tilde{\eta}_p^\circ)_*(\iota_Y)_* \alpha_Y.
\end{equation}
Toward calculating $(\tilde{\eta}_p^\circ)_*[\Zbar]$, we first calculate the second term on the right-hand side of \eqref{excess-formula}.

\begin{proposition}\label{excess.equiv}
We have $(\tilde{\eta}_p^\circ)_*(\iota_Y)_* \alpha_Y = \zeta_p + a_1 + gz \in A^1(\H_{3,g}'(3))$.
\end{proposition}
\begin{proof}
Let  $j_Y \colon Y \rightarrow \eta_q^*V_3$ be the restriction to $Y$ of the composition of the left-hand vertical maps in \eqref{diagram3-excess}.
 According to Proposition~\ref{prop:excess}, the cycle class $\alpha_Y$ is given by
\[\alpha_Y = \left\{c(j_Y^*N_{\iota_q}) c(N_{Y/S})^{-1}\right\}^1 \in A^1(Y),\]
where we write $\{-\}^k$ for the part of a Chow class in $A^k(Y)$.
Since the normal bundle of the diagonal $\Delta_{p = q}$ in $\widetilde{X}_3$ is naturally identified with $\eta_p^*T_{\P\E^\vee/\B} \otimes \O_{\Delta_{p = q}}$, we compute that
\begin{align*}
    (\tilde{\eta}_p^\circ)_* (\iota_Y)_* \alpha_Y &= (\tilde{\eta}_p^\circ \circ \iota_Y)_*\left(j_Y^* c_1(N_{\eta_q^*V_3/\eta_q^*(\U \times_\B X_3)}) - c_1(N_{\phi^{-1}(\Delta_{p = q})/S})\right) \\&= (\tilde{\eta}_p^\circ \circ \iota_Y)_*\left(j_Y^*\eta_q^* c_1(N_{V_3/(\U \times_\B X_3)}) - (\phi|_Y)^*c_1(N_{\Delta_{p = q}/\widetilde{X}_3})\right) \\
    &= (\tilde{\eta}_p^\circ \circ \iota_Y)_*\left(j_Y^*\eta_q^* c_1(\sP^2_{\pp \E^\vee/\cP}(\W)) - (\phi|_Y)^{*}c_1(\eta_p^*T_{\pp \E^\vee/\B} \otimes \O_{\Delta_{p = q}})\right) \\
    &= 3\zeta_p - (2\zeta_p - a_1 - gz) \in A^1(\H_{3,g}'(3)),
\end{align*}
where the last equality follows from the following computations. Applying the exact sequence \eqref{filt} together with the computation in Lemma~\ref{zv}, we have
\begin{align*}
    c_1(\sP^2_{\P\E^\vee/\cP}(\W)) &= c_1(\sP^1_{\P\E^\vee/\cP}(\W)) + c_1(\Sym^2\Omega_{\P\E^\vee/\cP} \otimes \W) \\
    &= (2c_1(\W) + c_1(\Omega_{\P\E^\vee/\cP})) + (2c_1(\Omega_{\P\E^\vee/\cP}) + c_1(\W)) \\
    &= 3(c_1(\W) + c_1(\Omega_{\P\E^\vee/\cP}))\\
    &= 3\zeta_p.
\end{align*}
Moreover, by the relative cotangent sequence $0 \to T_{\P\E^\vee/\cP} \to T_{\P\E^\vee/\B} \to \gamma^*T_{\cP/\B} \to 0$,
we have
\[c_1(T_{\P\E^\vee/\B}) = -c_1(\Omega_{\P\E^\vee/\cP}) - \gamma^* c_1(\Omega_{\cP/\B}) = 2\zeta_p - a_1 - gz. \qedhere\]
\end{proof}

Now the formula \eqref{excess-formula} can be applied to obtain the additional class $(\tilde{\eta}_p^\circ)_*[\Zbar] = [TT]$ on $A^1(\H_{3,g}'(3))$ that vanishes on $A^*(\H_{3,g}(3))$.

\begin{lemma}\label{zrel3-3}
We have $-\zeta_p - a_1 + 3a_2' - gz = 0 \in A^*(\H_{3,g}(3))$.
\end{lemma}
\begin{proof}
We will obtain this relation by using the fact that $0 = [TT] \in A^*(\H_{3,g}(3))$ and $[TT] = (\tilde{\eta}_p^\circ)_*[\Zbar]$.  In order to compute this pushforward, we first compute the left-hand side of \eqref{excess-formula}.

We denote by $F$ the composition $\H_{3,g}'(3) \rightarrow V_3 \rightarrow \U\times_{\B} X_3\rightarrow X_3$.
Recall that $\eta_q^*V_3$ is a subbundle of the vector bundle $\eta_p^*(\U \times_{\B} X_3)$ over $\widetilde{X}_3$. The quotient bundle of this inclusion is $\eta_q^*\mathscr{P}^2_{\pp \E^\vee/\cP}(\W)$.
As such, the fundamental class of $\eta_q^*V_3$ in $\eta_p^*(\U \times_{\B} X_3)$ is the pullback to $\eta_p^*(\U \times_{\B} X_3)$ of $\eta_q^*c_3(\sP^2_{\pp \E^\vee/\cP}(\W)) \in A^3(\widetilde{X}_3)$.

We will make use of the following diagram:
\begin{equation}\label{bigF}
    \begin{tikzcd}
\B                                & X_3 \arrow[l, "\pi \circ \gamma"']                  &                                                                &                                   & {\H'_{3,g}(3)} \arrow[lll, "F"']                               \\
X_3 \arrow[u, "\pi \circ \gamma"] & \widetilde{X}_3 \arrow[l, "\eta_q"'] \arrow[u, "\eta_p"] & \eta_p^*(\U\times_{\B} X_3) \arrow[l, "\text{v.b.}"'] & \eta_p^* V_3 \arrow[l, "\iota_p"'] & S \arrow[l, "\text{open}"'] \arrow[u, "\tilde{\eta}_p^\circ"'].
\end{tikzcd}
\end{equation}
Applying the push-pull formula along the outer square of \eqref{bigF} and using the exact sequence \eqref{filt}, we have
\begin{align*}
    &{\phantom{{}={}}} (\tilde{\eta}_p^\circ)_*(\iota_p\mid_S)^*(\iota_q)_*[\eta_q^* V_3] \\
    &= (\tilde{\eta}^\circ_p)_*(\iota_p\mid_S)^*\eta_q^*c_3(\sP^2_{\pp\E^\vee/\cP}(\W)) \\
    &= F^*(\pi\circ\gamma)^*(\pi\circ\gamma)_*c_3(\sP^2_{\pp\E^\vee/\cP}(\W)) \\
    &= F^*(\pi\circ\gamma)^*(\pi\circ\gamma)_* c_1(\Omega_{\pp \E^\vee/\cP}^{\otimes 2} \otimes \W) c_2(\sP^1_{\pp \E^\vee/\cP}(\W)) \\
    &= F^*(\pi\circ\gamma)^*(\pi\circ\gamma)_* c_1(\Omega_{\pp \E^\vee/\cP}^{\otimes 2} \otimes \W) c_1(\Omega_{\pp \E^\vee/\cP} \otimes \W) c_1(\W) \\
    &= F^*(\pi\circ\gamma)^*(\pi\circ\gamma)_*\left((2(-2\zeta_p + a_1 + (g+2)z) + 3\zeta_p - (a_1 + (g + 2)z) \right.\\
&\qquad \qquad \qquad \qquad \qquad \qquad \cdot    ((-2\zeta_p + a_1 + (g+2)z) + 3\zeta_p - (a_1 + (g + 2)z)) \\
&\qquad \qquad \qquad \qquad \qquad \qquad \cdot (3\zeta_p - (a_1 + (g + 2)z))
    \left.\right). \\
    \intertext{Expanding this as a polynomial in $\zeta_p$, we find the following:}
    &= F^*(\pi\circ\gamma)^*(\pi\circ\gamma)_*\left[-3\zeta_p^3 + 4(a_1 + (g+2)z)\zeta_p^2 - (a_1^2 + 2(g+2)a_1z +(g+2)^2z^2)\zeta_p\right] \\
    &= F^*(\pi\circ\gamma)^*(\pi\circ\gamma)_*\left[-3\zeta_p^3 + 4A\zeta_p^2 - A^2\zeta_p\right],
\end{align*}
where
\[A \coloneq c_1(\E) = a_1 + (g+2)z.\]
From here, setting $B \coloneq c_2(\E) = a_2 + a_2'z$,
and using the relation $\zeta_p^2 = A\zeta_p - B$,
which follows from the projective bundle formula in $A^*(\P \E^{\vee})$, the above chain of equalities reduces to
\[F^*(\pi \circ \gamma)^* (\pi \circ \gamma)_*(3B \zeta_p - AB).\]
Since $A$ and $B$ are pulled back under $\gamma$, and $\gamma_*(1) = 0$ whereas $\gamma_*(\zeta_p) = 1$, an application of the projection formula shows that the above equals
\[F^*(\pi \circ \gamma)^*\pi_*(3B) = 3F^*(\pi \circ \gamma)^*\pi_*(a_2 + a_2'z).\]
Similarly, since $a_2$ and $a_2'$ are pulled back under $\pi$, an analogous application of the projection formula shows that the above equals $3a_2'$.

In all, then, we have proven that the left-hand side of equation \eqref{excess-formula} equals $3a_2'$.  Via \eqref{excess-formula} and Proposition \ref{excess.equiv}, we thus obtain the following equation in $A^1(\H'_{3,g}(3))$:
\[[TT] = (\tilde{\eta}_p^\circ)_*[\overline{Z}] = (\tilde{\eta}_p^\circ)_*(\iota_p\mid_S)^*[\eta_q^* V_3] - (\tilde{\eta}_p^\circ \circ \iota_Y)_* \alpha_Y = 3a_2'- \zeta_p - a_1  - gz.\]
Since $[TT]$ vanishes on $A^*(\H_{3,g}(3))$ by definition, the relation in the statement of the lemma is proved.
\end{proof}

We now have all of the requisite ingredients to prove, at last, that the Chow ring of $\sH_{3,g}(3)$ is trivial.

\begin{lemma}\label{lem:3}
$A^*(\sH_{3,g}(3)) = \qq$.
\end{lemma}
\begin{proof}
Since the Chow rings of $\sH_{3,g}(3)$ and $\H_{3,g}(3)$ are isomorphic, it suffices to show that $A^*(\H_{3,g}(3)) = \qq$. As the relations obtained in Lemma \ref{rel3}, \ref{zrel3-1}, \ref{zrel3-2}, and \ref{zrel3-3} are linearly independent in terms of generators $\zeta_p$, $z$, $a_1$, and $a_2'$ of $A^*(\H_{3,g}(3))$, the result follows.
\end{proof}

\section{Conclusion}

The combination of Lemmas \ref{ckgp111}, \ref{lem:111}, \ref{ckgp21}, \ref{lem:21}, \ref{ckgp3}, and \ref{lem:3} proves Theorem~\ref{withramification}---that is, that all of the degree-3 Hurwitz spaces with a single marked ramification point have trivial Chow ring and the Chow--K\"unneth generation Property.  Since Lemma \ref{h2} shows that the same is true of the degree-2 Hurwitz spaces with a single marked ramification point, and Lemma \ref{lem:BDs} shows that the codimension-1 boundary strata in $\sHbar_{3,g}$ are the images of gluing maps from products of such degree-2 or degree-3 marked Hurwitz spaces, we deduce that all of the codimension-1 boundary strata in $\sHbar_{3,g}$ have trivial Chow ring.  As discussed in the introduction, when combined with excision and the fact that $A^2(\sH_{3,g}) =0$ by the results of \cite{CaLa:22}, this implies that $A^2(\sHbar_{3,g})$ is generated by the fundamental classes of closures of codimension-2 boundary strata; that is, Theorem~\ref{main} is proved.

\bibliographystyle{alpha}

\begin{thebibliography}{ACV03}

\bibitem[ACV03]{ACV:03}
Dan Abramovich, Alessio Corti, and Angelo Vistoli.
\newblock Twisted bundles and admissible covers.
\newblock {\em Comm. Algebra}, 31(8):3547--3618, 2003.
\newblock Special issue in honor of Steven L. Kleiman.

\bibitem[BS23]{BaSc:23}
Younghan Bae and Johannes Schmitt.
\newblock Chow rings of stacks of prestable curves {II}.
\newblock {\em J. Reine Angew. Math.}, 800:55--106, 2023.

\bibitem[CL22]{CaLa:22}
Samir Canning and Hannah Larson.
\newblock Chow rings of low-degree {H}urwitz spaces.
\newblock {\em J. Reine Angew. Math.}, 789:103--152, 2022.

\bibitem[CL24]{CaLa:23}
Samir Canning and Hannah Larson.
\newblock On the {C}how and cohomology rings of moduli spaces of stable curves.
\newblock {\em J. Eur. Math. Soc.}, 2024.
\newblock Published online first.

\bibitem[DL21]{DL}
Andrea Di~Lorenzo.
\newblock The {C}how ring of the stack of hyperelliptic curves of odd genus.
\newblock {\em Int. Math. Res. Not. IMRN}, (4):2642--2681, 2021.

\bibitem[DP15]{DePa:15}
Anand Deopurkar and Anand Patel.
\newblock The {P}icard rank conjecture for the {H}urwitz spaces of degree up to
  five.
\newblock {\em Algebra Number Theory}, 9(2):459--492, 2015.

\bibitem[EF09]{EF}
Dan Edidin and Damiano Fulghesu.
\newblock The integral {C}how ring of the stack of hyperelliptic curves of even
  genus.
\newblock {\em Math. Res. Lett.}, 16(1):27--40, 2009.

\bibitem[EH16]{EiHa:16}
David Eisenbud and Joe Harris.
\newblock {\em 3264 and all that---a second course in algebraic geometry}.
\newblock Cambridge University Press, Cambridge, 2016.

\bibitem[EH22]{EH}
Dan Edidin and Zhengning Hu.
\newblock The integral {C}how rings of the stacks of hyperelliptic
  {W}eierstrass points, 2022.

\bibitem[Ful98]{Ful:98}
William Fulton.
\newblock {\em Intersection theory}, volume~2 of {\em Ergebnisse der Mathematik
  und ihrer Grenzgebiete. 3. Folge. A Series of Modern Surveys in Mathematics
  [Results in Mathematics and Related Areas. 3rd Series. A Series of Modern
  Surveys in Mathematics]}.
\newblock Springer-Verlag, Berlin, second edition, 1998.

\bibitem[FV11]{FV}
Damiano Fulghesu and Filippo Viviani.
\newblock The {C}how ring of the stack of cyclic covers of the projective line.
\newblock {\em Ann. Inst. Fourier (Grenoble)}, 61(6):2249--2275, 2011.

\bibitem[HM82]{HaMu:82}
Joe Harris and David Mumford.
\newblock On the {K}odaira dimension of the moduli space of curves.
\newblock {\em Invent. Math.}, 67(1):23--88, 1982.
\newblock With an appendix by William Fulton.

\bibitem[Kee92]{Kee:92}
Sean Keel.
\newblock Intersection theory of moduli space of stable {$n$}-pointed curves of
  genus zero.
\newblock {\em Trans. Amer. Math. Soc.}, 330(2):545--574, 1992.

\bibitem[Lan24]{Landi}
Alberto Landi.
\newblock The integral {C}how ring of the stack of pointed hyperelliptic
  curves, 2024.

\bibitem[Lar23]{Lar:23}
Hannah Larson.
\newblock The intersection theory of the moduli stack of vector bundles on
  {$\mathbb{P}^1$}.
\newblock {\em Canad. Math. Bull.}, 66(2):359--379, 2023.

\bibitem[Mul23]{Mul:23}
Scott Mullane.
\newblock The {H}urwitz space {P}icard rank conjecture for {$d > g-1$}.
\newblock {\em Adv. Math.}, 430:Paper No. 109211, 27, 2023.

\bibitem[Vis89]{Vis:89}
Angelo Vistoli.
\newblock Intersection theory on algebraic stacks and on their moduli spaces.
\newblock {\em Invent. Math.}, 97(3):613--670, 1989.

\end{thebibliography}

\end{document}